\documentclass[11pt]{amsart}
\usepackage{amsmath,amsxtra,amssymb,color,latexsym,epsfig,amscd,amsthm,subfigure,fancybox,epsfig}
\usepackage[mathscr]{eucal}
\usepackage{bbm}
\usepackage{float}
\usepackage{graphicx}
\usepackage{enumerate}
\usepackage{enumitem}
\usepackage{epsfig}
\usepackage{epstopdf}
\usepackage{cases}
\usepackage{hyperref}

\usepackage{array}

\hypersetup{
    colorlinks=true,
    linkcolor=blue,
    filecolor=magenta,
    urlcolor=cyan,
    citecolor=blue,
}

\newtheorem{thm}{Theorem}[section]
\newtheorem {asp}{Assumption}[section]
\newtheorem{lm}{Lemma}[section]
\newtheorem{prop}{Proposition}[section]

\theoremstyle{definition}

\theoremstyle{remark}
\newtheorem{rem}{Remark}[section]

\numberwithin{equation}{section}

\newcommand{\eps}{\varepsilon}

\newcommand{\F}{\mathcal{F}}

\newcommand{\E}{\mathbb{E}}

\newcommand{\bs}{\mathbf{s}}

\newcommand{\Lom}{\mathcal{L}}

\newcommand{\N}{{\mathbb{Z}}_+}
\newcommand{\Q}{{\mathbb{Q}}}

\newcommand{\BX}{\bar{X}}
\newcommand{\BY}{\bar{Y}}
\newcommand{\BS}{\mathbf{S}}

\newcommand{\PP}{\mathbb{P}}

\newcommand{\R}{\mathbb{R}}

\numberwithin{equation}{section}
\newcommand{\1}{\boldsymbol{1}}

\newcommand{\wdt}{\widetilde}

\newcommand{\op}{{\mathcal L}}

\newcommand{\bed}{\begin{equation}}
\newcommand{\eed}{\end{equation}}
\newcommand{\bea}{\bed\begin{array}{rl}}
\newcommand{\eea}{\end{array}\eed}

\newcommand{\barray}{\begin{array}{ll}}
\newcommand{\earray}{\end{array}}

\newcommand{\dist}{{\rm dist}}

\def\bar{\overline}
\def\hat{\widehat}
\def\a.s{\text{\;a.s.\;}}

\def\bnu{\boldsymbol{\nu}}
\def\bmu{\boldsymbol{\nu}}
\def\bdelta{\boldsymbol{\delta}}

\def\beq{\begin{equation}}
	
	\def\eeq{\end{equation}}

\def\ben{\begin{enumerate}}
	
	\def\een{\end{enumerate}}

\def\beqar{\begin{eqnarray}}
	
	\def\eeqar{\end{eqnarray}}

\def\beqarr{\begin{eqnarray*}}
	
	\def\eeqarr{\end{eqnarray*}}

\begin{document}
\title{Stochastic nutrient-plankton models}

\author[A. Hening]{Alexandru Hening }
\address{Department of Mathematics\\
Texas A\&M University\\
Mailstop 3368\\
College Station, TX 77843-3368\\
United States
}
\email{ahening@tamu.edu}

\author[N. T. Hieu]{Nguyen Trong Hieu}
\address{University of Science, \\ Vietnam National University, Hanoi,
	334 Nguyen Trai St\\
	Hanoi,
	Vietnam}
\email{hieunguyentrong@gmail.com}

\author[D.H. Nguyen]{Dang Hai Nguyen }
\address{Department of Mathematics \\
University of Alabama\\
345 Gordon Palmer Hall\\
Box 870350 \\
Tuscaloosa, AL 35487-0350 \\
United States}
\email{dangnh.maths@gmail.com}

\author[N. N. Nguyen]{ Nhu Ngoc Nguyen }
\address{
	Department of Mathematics and Applied Mathematical Sciences\\
University of Rhode Island\\
5 Lippitt Road, Suite 200\\
Kingston, RI 02881-2018\\
United States}
\email{nhu.nguyen@uri.edu }

\maketitle

\begin{abstract}
We analyze plankton-nutrient food chain models composed of phytoplankton, herbivorous zooplankton and a limiting nutrient. These models have played a key role in understanding the dynamics of plankton in the oceanic layer. Given the strong environmental and seasonal fluctuations that are present in the oceanic layer, we propose a stochastic model for which we are able to fully classify the longterm behavior of the dynamics. In order to achieve this we had to develop new analytical techniques, as the system does not satisfy the regular disspativity conditions and the analysis is more subtle than in other population dynamics models.
\bigskip

\end{abstract}

\section{Introduction}\label{sec:int}

The oceans of the world are populated by small, free floating or weakly swimming, organisms called plankton. More specifically, plankton can be divided into phytoplankton, which are plants, and zooplankton, which are animals that consume the phytoplankton. These tiny organisms have a significant impact on the various food chains present in the oceans as they form the bottom of the food chains. In addition, they also seem to play a role in the Earth's carbon cycle. Because it is hard to empirically measure the amount of plankton, it is important to build simple mathematical models that will allow us to better understand the dynamics of plankton.

The analysis of mathematical models for plankton dynamics can be traced to Hallam \cite{H77, H772, H78}, who obtained stability and persistence results for nutrient controlled plankton models. Since then, people have studied the dynamics of models that include phytoplankton, zooplankton and a nutrient that is consumed by the phytoplankton. This nutrient can be regenerated due to the bacterial decomposition of dead phytoplankton and zooplankton. In this paper we assume that the nutrient recycling is instantaneous, and therefore neglect the time required to regenerate the nutrient from dead plankton.

In our model, which first appeared in \cite{W88} and was generalized in \cite{R93}, the limiting nutrient has a constant input concentration $N^0$ while the nutrient, phytoplankton and zooplankton have constant washout rates $D, D_1$ and $D_2$. It is important to include the washout rates because they describe the removal due to washout, sinking, or harvesting of biotic mass from the ecosystem.

The models we study can be seen as describing the dynamics of the zooplankton-phytoplankton-nutrient trio within lakes or oceans. Since water masses have nutrient residence times of years \cite{PR85}, one must consider the regeneration of nutrient due to bacterial decomposition of dead plankton. We assume that the zooplankton only feeds on phytoplankton and that parts of the dead phytoplankton and zooplankton are instantaneously recycled into the nutrient. We are then able to find two thresholds, which depend on the model parameters, that completely characterize the persistence or extinction of the two types of plankton.

Natural ecosystems will be influenced by random environmental fluctuations. These fluctuations will have a significant impact on the dynamics of the various species. As a result, in order to have a realistic model of the species dynamics in an ecosystem it is key to include environmental fluctuations in the mathematical framework. It is well known that environmental fluctuations can have a significant impact on the long term behavior: in certain cases coexistence can be reversed into extinction while in others extinction becomes coexistence \cite{BL16, HN20, HNS21}. A successful way to model environmental fluctuations has been via stochastic differential equations, and more generally, Markov processes \cite{C82, CE89, C00, ERSS13, EHS15,  LES03, SLS09, SBA11, BS09, BHS08, B18, HNC20}.

There are many ways in which one can model the environmental fluctuations that affect an ecological system. One way is by going from ordinary differential equations (ODE) to stochastic differential equations (SDE). This amounts to saying that the various birth, death and interaction rates in an ecosystem are not constant, but fluctuate around their average values according to some white noise. There is now a well established general theory of coexistence and extinction for these systems when they are in Kolmogorov form \cite{SBA11,HN16,HNC20} if strong dissipativity assumptions hold.
Furthermore, under the same type of dissipativity assumptions, the complete classification of the dynamics for three-species SDE Kolmogorov systems has been provided in \cite{HNS21}.  The dissipativity assumption is natural in many systems - intuitively it says that if one of the species has a high population density, the species has a strong drift towards $0$. The simplest two-dimensional example would be the predator-prey system
\begin{equation*}
	\begin{aligned}
dX(t) &= X(t)(a-b X(t)-cY(t))dt + \sigma_1 X(t) dW_1(t)\\
dY(t) &= Y(t)(-d - fY(t) + cX(t))dt + \sigma_2 Y(t) dW_2(t)
\end{aligned}
\end{equation*}
where the dissipativity is due to the intraspecific competition terms $-bX^2(t)dt$ and $-fY^2(t)$. However, plankton models do not satisfy these boundedness/dissipativity conditions and therefore provide a significant technical challenge.

The paper \cite{YYZ19} studies plankton dynamics and was the main inspiration for writing the present paper. The results from \cite{YYZ19} are interesting but they are not sharp and do not provide a full classification of the dynamics. We significantly generalize the results from \cite{YYZ19} and provide a complete characterization of the long term behavior of the system. We note that we restrict our analysis to SDE but the results can be easily generalized to SDE with switching.

The contributions and novelties of this work is two-fold.
First, we advance the study of food chain models in marine ecology by investigating systems under general formulation and proving a complete characterization of what will happen in the long time, for the first time.
Second, we introduce new techniques to classify longtime properties of stochastic differential equations in which dissipativity conditions, which are needed in existing works, do not hold.

The paper is organized as follows. In Section \ref{s:ma} we give the mathematical setup and the results. In Section \ref{s:sketch} we give a sketch of the proof, explain the main technical difficulties, and showcase the new techniques that have to be developed in order to analyze this system. The analysis of the extinction results, Theorems \ref{thm2} and \ref{thm3}, appears in Sections \ref{s:ext_1} and \ref{s:ext}, while the proof of the main persistence result, Theorem \ref{thm4}, is in Section \ref{s:pers}.

\subsection{Mathematical Setup and Results}\label{s:ma}

A natural deterministic model for a nutrient-plankton system is given, according to \cite{R93}, by

\begin{equation}\label{main_det}
	\begin{aligned}
\frac{dX}{dt}(t)&= \Lambda-\alpha_1X(t) - aY(t)X(t)+\alpha_4Y(t)+\alpha_5 Z(t)\\
\frac{dY}{dt}(t)&= aY(t)X(t)-bY(t)Z(t)-\alpha_2Y(t)\\
\frac{dZ}{dt}(t)&= bY(t)Z(t)-\alpha_3Z(t)\\
\end{aligned}
\end{equation}
where $(X(t), Y(t), Z(t))$ are the densities of the nutrient, phytoplankton and zooplankton at time $t\geq 0$. The various coefficients are related to biological factors as follows:
$N_0$ is the input concentration of nutrient, $\alpha_1$ is the washout rate for the nutrient, $\Lambda := N^0 \alpha_1$, $\alpha_2$ is the sum of the death rate and the washout rate of the phytoplankton, $\alpha_3$ is the sum of the death rate and the washout rate of the zooplankton,  $a$ is the maximal nutrient uptake rate of phytoplankton, $b$ is the maximal nutrient uptake rate of zooplankton,  $\alpha_4$ is the nutrient recycling rate from dead phytoplankton, and $\alpha_5$ is the nutrient recycling rate from dead zooplankton. This model has been studied in \cite{R93} where the author found sufficient conditions for extinction and persistence.
It is natural to generalize the functional responses from \eqref{main_det} so that nonlinear interactions can be captured. One way is by looking at the dynamics of the type
\begin{equation}\label{main0}
	\begin{aligned}
\frac{dX}{dt}(t) &= \Lambda- F_1(X(t),Y(t))X(t)Y(t)-\alpha_1X(t)+\alpha_4 Y(t)+\alpha_5 Z(t)\\
\frac{dY}{dt}(t) &= F_1(X(t),Y(t))X(t)Y(t)-F_2(Y(t),Z(t))Y(t)Z(t)-\alpha_2Y(t)\\
\frac{dZ}{dt}(t) &= F_2(Y(t),Z(t))Y(t)Z(t)-\alpha_3Z(t)\\
\end{aligned}
\end{equation}
where $F_1(x,y), F_2(y,z)$ can now be non-constant functions. The last step is introducing the environmental white-noise fluctuations, which turns the system of ODE \eqref{main0} into the system of SDE
\begin{equation}\label{main}
	\begin{aligned}
dX(t) &= [\Lambda- F_1(X(t),Y(t))X(t)Y(t)-\alpha_1X(t)+\alpha_4 Y(t)+\alpha_5 Z(t)]dt+\sigma_1 X(t)dW_1(t)\\
dY(t) &= [F_1(X(t),Y(t))X(t)Y(t)-F_2(Y(t),Z(t))Y(t)Z(t)-\alpha_2Y(t)]dt+\sigma_2 Y(t)dW_2(t)\\
dZ(t) &= [F_2(Y(t),Z(t))Y(t)Z(t)-\alpha_3Z(t)]dt+\sigma_3 Z(t)dW_3(t)\\
\end{aligned}
\end{equation}
where $(W_1(t), W_2(t), W_3(t))$ is a standard Brownian motion on $\R^3$ on a complete probability space $(\Omega,\F,\{\F_t\}_{t\geq0},\PP)$ with a filtration $\{\F_t\}_{t\geq 0}$ satisfying the usual conditions.
Throughout this paper, $\R^3_+=\{(x,y,z)\in\R^3: x,y,z\geq 0\}$,
$\R^{3,\circ}_+=\{(x,y,z)\in\R^3: x,y,z> 0\}$.
Let $\BS(t):=(X(t), Y(t), Z(t))$ and let $\bs=(x,y,z) \in \R_{+}^{3}$ denote the initial conditions, that is $\BS(0):=(X(0), Y(0), Z(0))=\bs$.
We denote by $\op$ the generator of the diffusion process $\BS$ from \eqref{main}.
We will also use $\PP_{\bs}$, $\E_\bs$ to indicate the initial value of the solutions.

The following assumption is held throughout the paper.
\begin{asp}\label{asp-1}
The following conditions hold.
\begin{enumerate}
	\item $\alpha_4<\alpha_2$ and $\alpha_5<\alpha_3$.
	\item $F_1(\cdot)$ and $F_2(\cdot)$ are  functions on $\R^2_+$ bounded by $L>0$. Suppose $F_1(u,v)u, F_2(u,v)u$ are Lipschitz functions whose Lipschitz coefficients are bounded by $L.$
	\item $F_1(u,0)u$ is nondecreasing.
\end{enumerate}
\end{asp}
\begin{rem}
	Assumption \ref{asp-1} (1) is natural since $\alpha_4$ (resp. $\alpha_5$) is the nutrient recycling rate from dead phytoplankton (resp. zooplankton) that must be always less than $\alpha_2$ (resp. $\alpha_3$), the death rate and the washout rate of the phytoplankton (resp. zooplankton).
	Assumption \ref{asp-1} (2) and (3) are mild and satisfied by almost all the models from the literature.
\end{rem}


The first theorem tells us that we can bound the moments of the process, and that the process stays in compact sets with large probability.
\begin{thm}\label{thm1}
	For any initial value $\bs=(x,y,z)\in\R^3_+$, there exists a unique a global solution
	$\BS(t)$ to \eqref{main}
	such that $\PP_{\bs}\{\BS(t)\in\R^3_+,\  \forall t\geq0\}=1.$
	Moreover, $X(t)>0$ for all $t>0$ with probability 1 and if $Y(0)=0$ (resp. $Z(0)=0$) then $Y(t)=0$ (resp. $Z(t)=0)$ for all $t\geq0$ with probability 1.
	We also have that  $\BS(t)$ is a Markov-Feller process on $\R^3$.
	Furthermore,
there are $q_0>1$, $\alpha_0>0$ such that for any $q\in[1,q_0]$,
\begin{equation}\label{e1-thm1}
	\E_\bs (1+X(t)+Y(t)+Z(t))^{q}\leq (1+x+y+z)^{q} e^{-\alpha_0 t} + C_{q_0},\;\forall \bs\in\R^3_+.
\end{equation}
In addition, there exists $\bar K>0$ such that
\begin{equation}\label{e2-thm1}
\E_\bs (1+X(t)+Y(t)+Z(t))^{2}\leq  e^{\bar K t} (1+x+y+z)^2,\;\forall \bs\in\R^3_+.
\end{equation}
Finally, for any $\eps>0, H>0, T>0$, there exists $\wdt K(\eps, H, T)>0$ such that
\begin{equation}\label{e3-thm1}
\PP_\bs\left\{X(t)+Y(t)+Z(t)\leq \wdt K(\eps, H, T),\;\forall 0\leq t\leq T\right\}\geq 1-\eps \text{ given } |\bs|\leq H.
\end{equation}
\end{thm}

Let $\hat X$ be the solution of the following equation
\begin{equation}\label{main11}
		d\hat X(t)= [\Lambda- \alpha_1 \hat X(t)]dt+\sigma_1 \hat X(t)dW_1(t).
\end{equation}
One can show that this one-dimensional SDE has a unique invariant measure $\mu_1$ on $[0,\infty)$, which is an inverse Gamma distribution (see Lemma \ref{lm5}). Define
$$\lambda_1:=\int_{[0,\infty)}F_1(u,0)u\mu_1(du) -\alpha_2-\frac{\sigma_2^2}{2}.$$

\begin{rem}
We note that when $F_1(\cdot,\cdot)=a$ is a constant function, and we are in the simplified setting where the deterministic part is the one from \eqref{main_det}, we get
\[
\lambda_1 = a\frac{\Lambda}{\alpha_1}-\alpha_2-\frac{\sigma_2^2}{2}.
\]
\end{rem}

The next result tells us that if $\lambda_1$ is negative then both the phytoplankton and the zooplankton go extinct with probability $1$. Furthermore, it also gives the exact exponential rates of extinction.
\begin{thm}\label{thm2}
	If $\lambda_1<0$ then for any $\BS(0)=\bs\in \R_{+}^{3,\circ}$ we have with probability 1 that
	\begin{equation}\label{thm2-e1}
\lim_{t\to\infty}\frac{\ln Y(t)}t=\lambda_1 \text{ and } \lim_{t\to\infty}\frac{\ln Z(t)}t=-\alpha_3-\frac{\sigma_3^2}2.
	\end{equation}
\end{thm}

We are wondering that what will happen if $\lambda_1>0$.
We consider a system in the absence of zooplankton that is given by
\begin{equation}\label{main2}
	\begin{aligned}
		d \bar X(t)&= [\Lambda- F_1( \bar X(t),\bar Y(t)) \bar X(t) \bar Y(t)-\alpha_1 \bar X(t)+\alpha_4  \bar Y(t)]dt+\sigma_1  \bar X(t)dW_1(t)\\
		d \bar Y(t)&= [F_1( \bar X(t),\bar Y(t)) \bar X(t) \bar Y(t)-\alpha_2\bar Y(t)]dt+\sigma_2  \bar Y(t)dW_2(t).
	\end{aligned}
\end{equation}
If $\lambda_1>0$, the next proposition shows that the phytoplankton-nutrient system \eqref{main2} is persistent and has a unique invariant probability measure.
We will use subscripts in $\E_{x,y}$ to indicate initial values of equation \eqref{main2}.
\begin{prop}\label{prop2}
		Let $(\BX,  \BY)$ be the solution to \eqref{main2}. If $\lambda_1>0$ then for sufficiently small $\theta>0$ there exist constants $K_\theta, \gamma_\theta>0$ such that
\begin{equation}\label{e0-prop2}
			\E_{x,y} [ (\BY(t))^{-\theta}]\leq K_\theta e^{-\gamma_\theta t}y^{-\theta} + K_\theta,\;\forall x\geq 0, y>0.
\end{equation}
As a result of the nondegeneracy of the diffusion process $(\BX(t),\BY(t))$,
	there exists a unique invariant measure $\mu_{12}$ of $(\BX(t),\BY(t))$  on $\R_{+}^{2,\circ}$.
\end{prop}
Therefore, if $\lambda_1>0$, we can define the invasion rate of the zooplankton into $\mu_{12}$ via $$\lambda_2:=\int_{\R_{12+}^{\circ}} F_2(u,v)u\mu_{12}(dudv)-\alpha_3-\frac{\sigma_3^2}2.$$
The normalized random occupation measure is given by
$$\wdt\Pi_t^\bs(\cdot):=\frac{1}{t}\int_0^t \1_{\{\BS(u)\in\cdot\}}\,du,$$
where the superscript $\bs$ indicates the corresponding initial condition.
Finally, we are able to show that if $\lambda_1>0$ then the sign of $\lambda_2$ determines the extinction/persistence of the zooplankton.

\begin{thm}\label{thm3}
If $\lambda_1>0$ and $\lambda_2<0$ then, for any $\bs\in\R^{3,\circ}_{+}$ with probability one
	\begin{equation}\label{thm3-e1}
		\lim_{t\to\infty}\frac{\ln Z(t)}t=\lambda_2
	\end{equation}
and with probability one the family of normalized random occupation measures $(\wdt\Pi_t^\bs)_{t> 0}$ converges weakly to $\mu_{12}$.
\end{thm}

\begin{thm}\label{thm4}
 If $\lambda_1>0$ and $\lambda_2>0$ then there exists a unique invariant measure $\mu^\circ$ on $\R^{3,\circ}_+$.
 Furthermore,
 for any $\bs\in\R^{3,\circ}_{+}$
	\begin{equation}\label{thm4-e1}
		\lim_{t\to\infty}t^{\wdt q-1}\|P_t(\bs,\cdot)-\mu^\circ(\cdot)\|_{TV}=0,\;\forall 1\leq \wdt q<q_0,
	\end{equation}
where $\|\cdot\|$ is the total variation metric and $P_t(\bs,\cdot)=\PP_\bs(\BS(t)\in\cdot)$ is the transition probability of the process $\BS(t)$.
\end{thm}
The complete characterization of the underlying system is summarized in the following table.
\begin{center}
	\begin{tabular}{|m{2.5cm}| m{12.5cm} |}
		\hline
		$\lambda_1<0$\newline
		(Theorem \ref{thm2}) & The phytoplankton $Y(t)$ and the zooplankton $Z(t)$ go extinct exponentially fast with probability $1$; the nutrient $X(t)$ converges weakly to the solution $\hat X(t)$ of \eqref{main11}.
		\\
		\hline
		$\lambda_1>0$, $\lambda_2<0$\newline (Theorem \ref{thm3}) & The zooplankton $Z(t)$ goes extinct exponentially fast with probability $1$; the nutrient-phytoplankton subsystem $(X(t),Y(t))$ converges weakly to the solution $(\bar X(t),\bar Y(t))$ of \eqref{main2}.
		 \\  \hline
		$\lambda_1>0$, $\lambda_2>0$\newline (Theorem \ref{thm4}) & Coexistence: the process $(X(t),Y(t),Z(t))$ has a unique invariant measure $\mu^\circ$ on $\R^{3,\circ}_+$, and the transition probability converges to $\mu^\circ$ with polynomial rate.  \\
		\hline
	\end{tabular}
\end{center}

\subsection{Sketch of proof, technical difficulties and novel approaches}\label{s:sketch}
	General results for extinction and persistence of Kolmogorov SDE systems appear in \cite{HN16}. However, those results cannot be applied to the nutrient-plankton model \ref{main}. This is because the dissipativity/boundedness condition \cite[(1.2) in Assumption 1.1]{HN16} is not satisfied for \ref{main}. This condition was used to prove that the process would return quickly into compact sets as well as the tightness of the random occupation measures. Because of this we had to develop new methods in order to get sharp extinction and persistence results.
We present the main ideas and difficulties for the proofs of Theorem \ref{thm3} and Theorem \ref{thm4} below.

The first ingredient in determining whether a species persists or goes extinct is looking at its long term growth rate at small densities. This is sometimes called the invasion rate. It turns out that these invasion rates can be computed as the external Lyapunov exponents, i.e. the log-growth rates averaged with respect to certain invariant measures which are supported on the boundary; see \cite{HN16, HNS21} for an exposition of the concept of invasion rate. For our models, the key invasion rates are $\lambda_1$ and $\lambda_2$ and we can show that extinction/persistence of the phytoplankton and the zooplankton is determined by the signs of $\lambda_1$ and $\lambda_2$.
Due to the lack of boundedness/dissipativity, we cannot obtain an exponential convergence rate in the case $\lambda_1>0$ and $\lambda_2>0$. Instead, we follow the techniques from \cite{benaim2022stochastic} to obtain a polynomial rate of convergence.

The hardest part is proving the extinction result (Theorem \ref{thm3}). Without the tightness of the family of random occupation measures $(\wdt\Pi_t^\bs(\cdot))_{t>0}$, the methods from \cite{HN16} do not work. We develop a new coupling method to compare the solution near the boundary (when $Z(t)$ is small) and the solution on the boundary (when $Z(t)=0$). While the comparison in a finite interval is standard, it is not sufficient to obtain the desired result which requires the two solutions to be close with a large probability in the infinite interval $[0,\infty)$.
	In order to overcome this obstacle, we construct a coupled system $(X(t), Y(t), \bar X(t), \bar Y(t), \bar Z(t))$ where $(X(t), Y(t))$ is the solution on the boundary $(Z(t)=0))$ and
	$(\bar X(t), \bar Y(t), \bar Z(t))$ has initial value close enough to the initial value of $(X(t), Y(t), Z(t))$. The process $(\bar X(t), \bar Y(t), \bar Z(t))$ after a change of measure is the solution to \eqref{main} up to a "separating" time $\tau$ and we show that the separating time is infinity with a large probability.

	Standard coupling methods often define $\tau$ as $\tau:=\inf\{t\geq 0: \bar Z(t)\geq \delta\}$ for some small $\delta$. However, this definition does not work on an infinite interval.
	Instead, we will define the ``separation" time $\tau$ as the time $Z(t)$ exceeds an exponential decay $\tau:=\inf\{t\geq0: \bar Z(t)\geq \delta e^{-\gamma_0 t}\}$. With this definition, it becomes much more difficult to show that $\tau=\infty$ with a large probability.
	The idea to tackle this difficulty is based on the strong correlation between $|X(t)-\bar X(t)|+|Y(t)-\bar Y(t)|$ being
	small and $\bar Z(t)$ decaying exponentially fast. If $|X(t)-\bar X(t)|+|Y(t)-\bar Y(t)|$ is small for a long time then
	$\bar Z(t)$ is still bounded by an exponential decay and when $\bar Z(t)$ is bounded by an exponential
	decay, one can establish a good bound for $|X(t)-\bar X(t)|+|Y(t)-\bar Y(t)|$ for the infinite interval $[0,\infty)$.	
\section{Proofs of Theorems \ref{thm1} and \ref{thm2}}\label{s:ext_1}

	\begin{proof}[Proofs of Theorem \ref{thm1}]
		 The existence and uniqueness of solutions can be proved similarly to \cite[Appendix B]{xu2016competition}. The proof for the Markov-Feller property of $(\BS(t))$ can be found in \cite{nguyen2017certain}.
	Therefore, the following is devoted to proofs of \eqref{e1-thm1}, \eqref{e2-thm1}, and \eqref{e3-thm1}.
		
Denote $\alpha_0:=\frac13\min\{\alpha_1,\alpha_2-\alpha_4,\alpha_3-\alpha_5\}$, and	let $q_0\in(1,2)$ be such that
	$(q_0-1)(\sigma_1^2\vee\sigma_2^2\vee\sigma_3^2)\leq \alpha_0$. Define
	$U^q(\bs)=(1+x+y+z)^q$. For $0<q\leq q_0$,
	we have
	\begin{equation}\label{eq-LUU}
	\begin{aligned}\op U^q(\bs)=&[U^q]_x(\bs)\big(\Lambda-F_1(x,y)xy-\alpha_1 x+\alpha_4 y+\alpha_5 z\big)\\
		&+[U^q]_y(\bs)\big(F_1(x,y)xy-F_2(y,z)yz-\alpha_2y\big)\\
		&+[U^q]_z(\bs)\big(F_2(y,z)yz-\alpha_3z \big)\\
		&+[U^q]_{xx}(\bs)\frac{\sigma_1^2x^2}2+[U^q]_{yy}(\bs)\frac{\sigma_2^2y^2}2+[U^q]_{zz}(\bs)\frac{\sigma_3^2z^2}2\\
		\leq &q\big(\Lambda-\alpha_1x-(\alpha_2-\alpha_4)y-(\alpha_3-\alpha_5)z\big)(1+x+y+z)^{q-1}\\
		&+\frac{q(q-1)}2(1+x+y+z)^{q-2}(\sigma_1^2x^2+\sigma_2^2y^2+\sigma_3^2y^2)		
			\\
			\leq& q\left(\Lambda(1+x+y+z)^{q-1}-2\alpha_0(1+x+y+z)^{q}\right).
		\end{aligned}
		\end{equation}
Since for any $0\leq q\leq q_0$, $$C_q:=\sup_{\bs=(x,y,z)\in\R_+^3}q\left(\Lambda(1+x+y+z)^{q-1}-(2-q)\alpha_0(1+x+y+z)^{q}\right)<\infty,$$ we obtain that
\begin{equation}\label{eq-LU}
\op U^q(\bs)\leq C_{q}-q\alpha_0U^q(\bs), \,\forall \bs\in\R^3_+.
\end{equation}
Let $\bar \tau_n=\inf\{t\geq 0:U(\BS(t))\geq n\}$.
Because of \eqref{eq-LU} and It\^o's formula, we have
	\begin{equation}\label{eq-222}
	\begin{aligned}
	\E_\bs e^{\alpha_0 (t\wedge\bar\tau_n)}U^q(\BS(t\wedge\bar\tau_n))\leq& U^q(\bs)+\E_\bs\left(\int_0^{t\wedge\bar\tau_n}C_{q_0}e^{q\alpha_0s}ds\right)\\
	\leq& U^q(\bs)+ C_{q_0}\int_0^te^{\alpha_0s}ds\\
\leq& U^q(\bs)+ \frac{C_{q_0}}{\alpha_0}e^{\alpha_0t}.
	\end{aligned}
\end{equation}
Dividing both sides of \eqref{eq-222} by $e^{q\alpha_0 t}$ and letting $n\to\infty$, we obtain \eqref{e1-thm1}.

Similarly, with some elementary estimates as the process of getting \eqref{eq-LUU}, we have
\begin{equation}\label{eq-LU2}
[\op U^2](\bs)\leq \bar K\, U^2, \forall \bs\in\R^3_+.
\end{equation}
Thus, from \eqref{eq-LU2} and Dynkin's formula, we get \begin{equation}\label{e6-thm1}
	\begin{aligned}
	\E_\bs e^{-\bar K t}U^2(\BS(t\wedge\bar\tau_n))\leq 	\E_\bs e^{-\bar K (t\wedge\bar\tau_n)}U^2(\BS(t\wedge\bar\tau_n))\leq& U^2(\bs).
	\end{aligned}
\end{equation}
Letting $n\to\infty$, we can derive \eqref{e2-thm1} from Lebesgue’s dominated convergence theorem. \eqref{e3-thm1} can also be obtained easily from \eqref{e6-thm1}.
\end{proof}

The remaining of this section is devoted to the proof of Theorem \ref{thm2}. We start with the following auxiliary Lemmas \ref{lm4} and \ref{lm5}, and Proposition \ref{prop3}. The first lemma establishes some estimates (in probability) for $\ln Y(t)$ and $\ln Z(t)$ in finite time intervals given initial conditions belonging in a compact set. The second lemma states the ergodicity of the process on the boundary corresponding $y=z=0$.
Proposition \ref{prop3} will show that if $\lambda_1<0$ and solutions start with small $Y(0)$ and $Z(0)$, then $Y(t)$ and $Z(t)$ converges to $0$ (exponentially fast) with large probability.
\begin{lm}\label{lm4}
For any $\eps>0$, $H>0$ and $T>0$, there exists $K_{\eps,H, T}>0$ such that
$$
\PP_{\bs}\left\{|\ln Z(t)-\ln z|\vee |\ln Y(t)-\ln y|\leq K_{\eps, H, T},\;\forall  0\leq t\leq T\right\}\geq 1-\eps \text{ if } \bs\in[0,H]\times(0,1)^2.
$$
	\end{lm}
\begin{proof}
In view of \eqref{e3-thm1}, there exists $K_1:=K_1(\eps, H, T)$ such that
\begin{equation}\label{e2-lm4}
\PP_{\bs} \left\{F_1(X(t), Y(t))X(t)+F_2(Y(t), Z(t))Y(t)\leq K_1\text{ for all } 0\leq t\leq T\right\}\geq 1-\frac\eps2 \text{ given } \bs\in[0,H]\times(0,1)^2.
\end{equation}
Moreover, there is $K_2:=K_2(\eps, T)>0$ such that
\begin{equation}\label{e3-lm4}
	\PP \left\{|\sigma_2W_2(t)|+|\sigma_3W_3(t)|\leq K_2\text{ for all } 0\leq t\leq T\right\}\geq 1-\frac\eps2.
\end{equation}
On the other hand, we deduce from It\^o's formula that
\begin{equation}\label{e4-lm4}
	\begin{cases}
		\ln Y(t)-\ln y=&\int_0^t F_1(X(s), Y(s))X(s)ds-\int_0^t F_2(Y(s), Z(s))Z(s)ds-(\alpha_2+\frac{\sigma_2^2}2)t+\sigma_2W_2(t),\\
		\ln Z(t)-\ln z=&\int_0^t F_2(Y(s), Z(s))Y(s)ds-(\alpha_3+\frac{\sigma_3^2}2)t+\sigma_3W_3(t).
	\end{cases}
\end{equation}
Applying \eqref{e2-lm4} and \eqref{e3-lm4} into \eqref{e4-lm4} we can easily obtain the desired result.
\end{proof}
\begin{lm}\label{lm5}
	For any $\theta>0$, any initial condition $x\geq 0$, there exists a unique solution ${\hat X}^\theta_x(t)$ to
	\begin{equation}\label{main-theta}
			d{\hat X}^\theta(t)=[\Lambda+\theta- \alpha_1{\hat X}^\theta(t)]dt+\sigma_1 {\hat X}^\theta(t)dW_1(t),\quad {\hat X}^\theta(0)=x.
	\end{equation}
The solution process ${\hat X}^\theta$ has a unique invariant probability measure $\mu^\theta$ on $[0,\infty)$, which is an inverse Gamma distribution with density $g_\theta(u)= \frac{\beta_\theta^\alpha}{\Gamma(\alpha)} u^{-\alpha-1}\exp\left(-\frac{\beta_\theta}{u}\right), u>0$, with $\alpha=1+2\frac{\alpha_1}{\sigma_1^2}$ and $\beta_\theta=\frac{2(\Lambda+\theta)}{\sigma_1^2}$. In particular $\int_{[0,\infty)}u\mu^\theta(du)=\frac{\beta_\theta}{\alpha-1}=\frac{\Lambda+\theta}{\alpha_1}$. Furthermore, $r(u)=u^{q_0}$ is $\mu^\theta$-integrable.

Note that $\mu^0=\mu_1$, is the unique invariant probability measure of \eqref{main11}. Define
$$
\ell_\theta:=\int_{[0,\infty)}F_1(u,0)u\mu^\theta(du),
$$
then we have
\begin{equation}\label{eq}
\lim_{\theta\to 0^+}\ell_\theta=\int_{[0,\infty)}F_1(u,0)u\mu_1(du)=\ell_0=\lambda_1+\alpha_2+\frac{\sigma_2^2}{2}.
\end{equation}
\end{lm}
\begin{proof}
	The proof is almost identical to that of \cite[Lemma 4.1]{nguyen2020long} and is therefore omitted.
\end{proof}

\begin{prop}\label{prop3}
	Suppose $\lambda_1<0$.
	For any $H>0$ and $\eps\in(0,1)$, there exists $\delta=\delta(\eps,H)>0$ such that
\begin{equation}\label{e0-prop3}
	\PP_{\bs}\left\{\lim_{t\to\infty}\frac{\ln Y(t)}{t}=\lambda_1 \text{ and } \lim_{t\to\infty}\frac{\ln Z(t)}{t}=-\alpha_3-\frac{\sigma_3^2}2\right\}\geq 1-\eps \text{ given } \bs\in[0,H]\times(0,\delta)^2.
\end{equation}
\end{prop}
\begin{proof}
	Let $\Delta_0=\frac13\left(\left(\alpha_3+\frac{\sigma^2_3}2\right)\wedge|\lambda_1|\right)$.
	In view of Lemma \ref{lm5} (and \eqref{eq}), we can choose (and then fix) a $\theta\in(0,\alpha_4\wedge\alpha_5\wedge \frac{\alpha_4\Delta_0}L)$ such that
	$$
	\ell_\theta\leq \lambda_1+\Delta_0+\alpha_2+\frac{\sigma_2^2}{2}.
	$$
Define $\xi:=\inf\{t\geq 0: \alpha_4Y(t)+\alpha_5Z(t)\geq\theta\}$.
Because of standard comparison theorems \cite{IW1}, we have
that given $X(0)=x\in[0,H]$,
$X(t)\leq X_H^\theta(t)\,\forall 0\leq t\leq \xi$ with probability 1.
For $t\leq \xi$, we have
\begin{equation}\label{e2-prop3}
	\begin{aligned}
\ln Y(t)-\ln y=&\int_0^t (F_1(X(s), Y(s))X(s)-
F_2(Y(s), Z(s))Z(s))ds-\left(\alpha_2+\frac{\sigma_2^2}2\right)t+\sigma_2W_2(t)\\	
\leq &\int_0^t F_1(X(s), 0)X(s)ds-\left(\alpha_2+\frac{\sigma_2^2}2\right)t+\sigma_2W_2(t)\\
&+\int_0^t [F_1(X(s), Y(s))X(s)-F_1(X(s),0)X(s)]ds\\
\leq &	\int_0^t F_1(X_H^\theta(t), 0)X_H^\theta(t)ds-\left(\alpha_2+\frac{\sigma_2^2}2\right)t+\sigma_2W_2(t)+\frac{L \theta}{\alpha_4}t;
\end{aligned}
	\end{equation}
	and
\begin{equation}\label{e3-prop3}
	\begin{aligned}
				\ln Z(t)-\ln z=&\int_0^t F_2(Y(s), Z(s))Y(s)ds-(\alpha_3+\frac{\sigma_3^2}2)t+\sigma_3W_3(t)\\
				\leq& L\int_0^t Y(s)ds-(\alpha_3+\frac{\sigma_3^2}2)t+\sigma_3W_3(t)\\
				\leq & \frac{L\theta}{\alpha_4}-(\alpha_3+\frac{\sigma_3^2}2)t+\sigma_3W_3(t)\\
				\leq & -2\Delta_0t+\sigma_3W_3(t).
			\end{aligned}
\end{equation}
In view of the ergodicity of ${\hat X}^\theta$ and law of large numbers for martingales,  we can find $T>0$ and a set
 $\wdt\Omega_1\subset\Omega$ such that $\PP\{\wdt\Omega_1\}\geq 1-\frac{\eps}2$ and
for $\omega\in\wdt\Omega_1$, we have the following two estimates:
\begin{equation}\label{eq-111}
\int_0^t F_1({\hat X}^\theta_H(s),0){\hat X}^\theta_H(s)ds+\sigma_2W_2(t)\leq (\ell_\theta+\Delta_0)t\leq \big(\lambda_1+\alpha_2+\frac{\sigma_2^2}{2}+2\Delta_0\big)t,\; t\geq T,
\end{equation}
and
\begin{equation}\label{eq-112}
\sigma_3W_3(t)\leq \Delta_0t,\; t\geq T.
\end{equation}
In view of Lemma \ref{lm4}, for any $\eps>0$, $H>0$, $T>0$, we can choose $\bar K=\bar K_{\eps, H, T}$ such that
$\PP_\bs(\wdt\Omega_2)\geq 1-\eps$ given $\bs\in[0,H]\times[0,1]^2$ where
$$\wdt\Omega_2:=\left\{|\ln Z(t)-\ln z|\vee |\ln Y(t)-\ln y|\leq \bar K,\;\forall t\in[0,T]\right\}.$$
Let $\delta=\frac{\theta}{3(\alpha_4\vee\alpha_5)}e^{-\bar K}.$
Then, for $\omega\in\wdt\Omega_2$ and $y\vee z\leq\delta$, we have \begin{equation}\label{e6-prop3}Y(t)\leq ye^{\bar K}\leq \frac{\theta}{3(\alpha_4\vee\alpha_5)}\text{ and }Z(t)\leq ze^{\bar K}\leq \frac{\theta}{3(\alpha_4\vee\alpha_5)}\text{ for any }t\leq T.
\end{equation}
As a result, we must have $\xi>T$ for $\omega\in\wdt\Omega_2$.

Now, considering $y\vee z\leq\delta, x\leq H$ and $\omega\in\wdt\Omega_1\cap\wdt\Omega_2$, we have from \eqref{e2-prop3} and \eqref{eq-111} that
\begin{equation}\label{e4-prop3}
\ln Y(t)\leq \ln y +(\lambda_1+2\Delta_0)t\leq\ln y-\Delta_0t\leq\ln y\leq\ln \delta,
 \text{ for } T\leq t\leq\xi,
\end{equation}
and from \eqref{e3-prop3} and \eqref{eq-112} that
\begin{equation}\label{e5-prop3}
\ln Z(t)\leq \ln z -\Delta_0t\leq\ln z\leq\ln\delta,
\text{ for } T\leq t\leq\xi.
\end{equation}
As a result of \eqref{e6-prop3}, \eqref{e4-prop3}, \eqref{e5-prop3} and definition of $\delta$, $\alpha_4Y(t)+\alpha_5Z(t)<\theta$ for any $0\leq t\leq\xi$ and  $\omega\in\wdt\Omega_1\cap\wdt\Omega_2.$
Therefore, we must have $\xi=\infty$.

Given that $\xi=\infty$ in $\wdt\Omega_1\cap\wdt\Omega_2.$, we can see from \eqref{e4-prop3} and \eqref{e5-prop3} that
$$
\limsup_{t\to\infty}\frac{\ln Y(t)}t\leq \lambda_1+2\Delta_0\leq-\Delta_0<0 \text{ and }
\limsup_{t\to\infty}\frac{\ln Z(t)}t\leq -\Delta_0<0, \omega\in \wdt\Omega_1\cap\wdt\Omega_2.
$$
These limits imply that there is no invariant measure on $\R^{3,\circ}_+$.
By a similar proof or a reference to \cite[Theorem 2.2]{nguyen2020long}, there is no invariant measure on $\R^\circ_{12+}$ either.
As a result, $\bmu_1:=\mu_1\times\bdelta^*\times\bdelta^*$  is the unique invariant measure of $\{\BS(t)\},$ where $\bdelta^*$ is the Dirac measure with mass at $0$.

On the other hand, with probability 1, any weak-limit (if it exists) of $\wdt\Pi_t^\bs$ $(:=\frac{1}{t}\int_0^t \1_{\{\BS(u)\in\cdot\}}\,du)$ as $t\to\infty$ is a unique invariant measure of  $\{\BS(t)\}$; see e.g. \cite[Theorem 4.2]{EHS15}.
For $y\vee z\leq\delta, x\leq H$ and $\omega\in\wdt\Omega_1\cap\wdt\Omega_2$, because $\lim_{t\to\infty} (Y(t)+Z(t))=0$
and
\begin{equation}\label{e7-prop3}
\limsup_{t\to\infty}\frac1t\int_0^t X^{q_0}(s)ds\leq \lim_{t\to\infty}\frac1t\int_0^t ({\hat X}^{\theta}_H(s))^{q_0}ds=\int_{[0,\infty)}u^{q_0}\mu^\theta(du)<\infty,
\end{equation}
we get that $\{\wdt\Pi_t^\bs\}$ is tight for $\omega \in\wdt\Omega_1\cap\wdt\Omega_2$
and subsequently, its limit must be $\bmu_1$, the unique invariant probability measure on $\R^{3}_+$.
This weak convergence together with the integrability \eqref{e7-prop3} imply that
\begin{equation}\label{e8-prop3-1}\lim_{t\to\infty}\left(\frac1t\int_0^t F_1(X(s), Y(s))X(s)ds-\left(\alpha_2+\frac{\sigma_2^2}2\right)\right)=\int_{\R^3_+} F_1(u,v)u \mu_1(du,dv)-\left(\alpha_2+\frac{\sigma_2^2}2\right)=\lambda_1<0.
\end{equation}
Applying \eqref{e8-prop3-1} into \eqref{e2-prop3} we obtain that
\begin{equation}\label{e8-prop3-2}
\lim_{t\to\infty}\frac{\ln Y(t)}t=\lambda_1<0, \omega\in \wdt\Omega_1\cap\wdt\Omega_2,\text{ for any }\bs=(x,y,z)\text{ with } 0\leq x\leq H, y+z\leq\delta.
\end{equation}
Because $Y(t)$ tends to $0$ at the exponential rate $\lambda_1$, we have from the first equality of \eqref{e3-prop3} and the boundedness of $F_2$ that
\begin{equation}\label{e8-prop3-3}
\lim_{t\to\infty}\frac{\ln Z(t)}t=\lim_{t\to\infty}\frac1t\int_0^t F_2(Y(s), Z(s))Y(s)ds-(\alpha_3+\frac{\sigma_3^2}2)+\lim_{t\to\infty}\frac{\sigma_3W_3(t)}t=-(\alpha_3+\frac{\sigma_3^2}2).
\end{equation}	
The proof is complete.
\end{proof}
Now, we are ready to prove Theorem \ref{thm2}.
\begin{proof}[Proof of Theorem \ref{thm2}.]
Note again that $\bmu_1$ $(:=\mu_1\times\bdelta^*\times\bdelta^*)$
	is the unique invariant probability measure on the boundary and therefore,
	the only invariant probability measure in $\R^3_+$ because $(X(t), Y(t), Z(t))$ has no invariant probability measure in $\R^{3,\circ}_+$.
	Let $H$ be sufficiently large such that $\mu_1((0,H))>1-\eps$ and then let $\delta>0$ satisfy \eqref{e0-prop3}.
	
	Thanks to Theorem \ref{thm1}, the family
	$\left\{\check \Pi_t^{\bs}(\cdot):=\frac1t\int_0^t\PP_{\bs}\left\{(X(u),Y(u),Z(u))\in\cdot\right\}du, t\geq 0\right\}$
	is tight in $\R^3_+$.
	Since any weak-limit of $\check\Pi_t^{\bs}$ as $t\to\infty$ must be an invariant probability measure of $\{\BS(t)\}$, (see e.g. \cite[Theorem 9.9]{ethier2009markov}),
	we have that
	$\check\Pi_t^{\bs}$ converges weakly to $\bmu_1$ $(=\mu_1\times\bdelta^*\times\bdelta^*)$ as $t\to\infty$.
		Thus,  there exists a  $\check T=\check T(\bs,\eps)>0$ such that
	$$\check\Pi^{\check T}_{\bs}((0,H)\times(0,\delta)\times(0,\delta))>1-\eps,$$
	or equivalently,
	$$\dfrac1{\check T}\int_0^{\check T}\PP_{\bs}\{(X(t),Y(t),Z(t))\in (0,H)\times(0,\delta)\times(0,\delta)\}dt>1-\eps.$$
	As a result,
	$$\PP_{\bs}\{\hat\tau\leq\check T\}>1-\eps,$$
	where $\hat\tau=\inf\{t\geq 0: (X(t),Y(t),Z(t))\in (0,H)\times(0,\delta)\times(0,\delta)\}$.
	Using the strong Markov property and \eqref{e0-prop3},
	we deduce that
	\begin{equation}
	\begin{aligned}
		\PP_{\bs}&\left\{\lim_{t\to\infty}\frac{\ln Y(t)}{t}=\lambda_1 \text{ and } \lim_{t\to\infty}\frac{\ln Z(t)}{t}=-\alpha_3-\frac{\sigma_3^2}2\right\}\\
		&\geq (1-\eps)(1-\eps)>1-2\eps, \text{ for all } \bs\in\R^{3,\circ}_+.
		\end{aligned}
	\end{equation}
	Letting $\eps\to0$ we obtain the desired result.
\end{proof}

\section{Proof of Theorem \ref{thm3}}\label{s:ext}
We begin with a proof for Proposition \ref{prop2}.
\begin{proof}[Proof of Proposition \ref{prop2}]
		Let \begin{equation}\label{deltan*}
	\Delta_1:= \frac{\lambda_1}5>0, \text{ and }n^* \text{ be the smallest integer satisfying } \Delta_1(n^*-1)\geq \alpha_2+\frac{\sigma_2^2}2.
	\end{equation}	
		Because $\int_{\R_+}F_1(u,0)u\mu_1(du)=\lambda_1+\alpha_2+\frac{\sigma_2^2}2$ and $F_1(u,0)u$ is an increasing function we have $\lim_{u\to\infty} F_1(u,0)u> \lambda_1+\alpha_2+\frac{\sigma_2^2}2$.
Moreover,
		there exist $M>\lambda_1+\alpha_2+\frac{\sigma_2^2}2$ such that
	\begin{equation}\label{eq-F1M}
	\int_{\R_+}\wdt F_{1,M}(u,0)\mu_1(du)\geq \lambda_1+\alpha_2+\frac{\sigma_2^2}2-\frac{\Delta_1}2 \text{ where }\wdt F_{1,M}(u,v):=(F_1(u,v)u)\vee M,
	\end{equation}
and $H>0$ such that
\begin{equation}\label{eq-H}
F_{1}(u,0)u\geq \lambda_1+\alpha_2+\frac{\sigma_2^2}2-\Delta_1\text{ for any } u\geq H.
\end{equation}
From \eqref{eq-F1M} and the ergodicity of $\hat X$,
we have
$$
\lim_{t\to\infty}\E\frac1t\int_0^t \wdt F_{1, M}(\hat X_{0}(s), 0)ds\geq \lambda_1+\alpha_2+\frac{\sigma_2^2}2-\frac{\Delta_1}2,
$$
 where $\hat X_0(s)$ is the solution to \eqref{main11} with initial condition $0$.
As a result,
there exists $T>0$ such that
$$
\E\frac1t\int_0^t \wdt F_{1, M}(\hat X_{0}(s), 0)ds\geq \lambda_1+\alpha_2+\frac{\sigma_2^2}2-\Delta_1,\;\forall t\geq T.
$$
Because of the uniqueness of the solution, $\hat X_x(s)\geq\hat X_0(s), s\geq 0$ almost surely for any $x\geq 0$, where $\hat X_x(s)$ is the solution to \eqref{main11} with initial condition $x$.
Then, thanks to the monotone increasing property of $\wdt F_{1,M}(u,0)$ (inherited from that property of $F_1(u,0)u$), we have
$$
\E\frac1t\int_0^t \wdt F_{1,M}(\hat X_{x}(s), 0)ds\geq \lambda_1+\alpha_2+\frac{\sigma_2^2}2-\Delta_1,\;\forall t\geq T, x\geq 0.
$$

Note that $(\hat X(t),0)$ is the solution $ (\BX,\BY)$ to \eqref{main2} with initial value $\BY(0)=0$.		
Because of the Feller-Markov property of $(\BX,\BY)$, there exists $0<\delta_0<\frac{\Delta_1}L$ such that
for any $(x,y)\in[0,H]\times(0,\delta_0]$, we have
\begin{equation}\label{eq-33}
\E_{x,y}\frac1t\int_0^t \wdt F_{1,M}(\BX(s), \BY(s))\BX(s)ds\geq \lambda_1+\alpha_2+\frac{\sigma_2^2}2-2\Delta_1, \;\forall T\leq t\leq n^*T,
\end{equation}
where the subscript in $\E_{x,y}$ indicates the initial condition of $(\BX,\BY)$.

Now, let
$$	\phi_{x,y,t}(\theta):=\ln\E_{x,y}\exp\left\{-\theta \Big(\int_0^t\wdt F_{1,M}(\BX(s),\BY(s))\BX(s)ds-\alpha_2t-\frac{\sigma_2^2}2t+\sigma_2W_2(t)\Big)\right\},$$
be the log-Laplace transform of the random variable
$$
-\Big(\int_0^t\wdt F_{1,M}(\BX(s),\BY(s))\BX(s)ds-\alpha_2t-\frac{\sigma_2^2}2t+\sigma_2W_2(t)\Big).
$$
Because of the boundedness of $\wdt F_{1,M}$, by a property of the log-Laplace transform, see \cite[Lemma 3.5]{HN16},
we have that the $\phi_{x,y,t}(\theta)$ is twice differentiable (in $\theta$) on $[0,\frac 12)$, with
\begin{equation}\label{eq-34}
\dfrac{d\phi_{x,y,t}}{d\theta}(0)=\E_{x,y}\left\{- \Big(\int_0^t\wdt F_{1,M}(\BX(s),\BY(s))\BX(s)ds-\alpha_2t-\frac{\sigma_2^2}2t+\sigma_2W_2(t)\Big)\right\},
\end{equation}
and
\begin{equation}\label{eq-35}
\sup_{|\theta|<1, t\leq n^*T}{\dfrac{d^2\phi_{x,y,t}}{d\theta^2}(0)}\leq K_\phi,
\end{equation}
for some constant $K_\phi=K_\phi(M, n^*T)$.
Because of \eqref{eq-33} and \eqref{eq-34}, one has
\begin{equation}\label{eq-36}
\dfrac{d\phi_{x,y,t}}{d\theta}(0)\leq-\left( \lambda_1-2\Delta_1\right)t.
\end{equation}
From \eqref{eq-35} and \eqref{eq-36}, we can have a Taylor expansion as follows
\begin{equation}\label{eq-37}
	\begin{aligned}
\phi_{x,y,t}(\theta)
	\leq& \phi_{x,y,t}(0)+\theta\dfrac{d\phi_{x,y,t}}{d\theta}(0)+\theta^2\sup_{|\theta|<1}{\dfrac{d^2\phi_{x,y,t}}{d\theta^2}(0)}\\\leq &
	0-\theta\left(\lambda_1-2\Delta_1\right)t+\theta^2 K_\phi,\;\forall t\in[T, n^*T].
	\end{aligned}
\end{equation}
Because $\lambda_1-2\Delta_1\geq 3\Delta_1$, we can pick a $\theta>0$ such that
\begin{equation}\label{eq-theta}
-\theta\left(\lambda_1-2\Delta_1\right)T+\theta^2 K_\phi\leq -2\Delta_1\theta \text{ and } \theta<\frac{\Delta_1}{\sigma^2_2}.
\end{equation}
With this chosen $\theta$, we have from \eqref{eq-37} that $\phi_{x,y,t}(\theta)\leq -2\Delta_1\theta t,\;\forall t\in[T,n^*T]$. This implies
\begin{equation}\label{e4-prop2}
	\begin{aligned}
		\frac{\E_{x,y} [(\bar Y(t))^{-\theta}]}{y^{-\theta}}=&\E_{x,y}\exp\left\{-\theta \Big(\int_0^tF_{1}(\BX(s),\BY(s))\BX(s)ds-\alpha_2t-\frac{\sigma_2^2}2t+\sigma_2W_2(t)\Big)\right\}\\
			\leq &
	\E_{x,y}\exp\left\{-\theta\Big( \int_0^t\wdt F_{1,M}(\BX(s),\BY(s))\BX(s)ds-\alpha_2t-\frac{\sigma_2^2}2t+\sigma_2W_2(t)\Big)\right\}\\
	=& \exp(\phi_{x,y,t}(\theta))\leq e^{-2\Delta_1\theta T},\;\forall t\in[T, n^*T].
	\end{aligned}
\end{equation}
On the other hand,
note that
$
F_1(u,0)u\geq \lambda_1+\alpha_2+\frac{\sigma_2^2}2-\Delta_1, u\geq H
$ and
 $|F_{1}(u,v)u-F_{1}(u,0)u|\leq Lv$,
imply
\begin{equation}\label{eq-310}
	F_{1}(u,v)u\geq \lambda_1+\alpha_2+\frac{\sigma_2^2}2-2\Delta_1 \text{ if } u\geq H \text{ and } v\leq\frac{\Delta_1}L.
\end{equation}
Because of \eqref{eq-310}, \eqref{deltan*}, and $\theta<\frac{\Delta_1}{\sigma_2^2}$ (due to \eqref{eq-theta}),
we have
\begin{equation}\label{e3-prop2}
	\begin{aligned}
d(\BY(t))^{-\theta}=&-\theta (\BY(t))^{-\theta}\left(F_1(\BX(t),\BY(t))\BX(t)-\alpha_2-(\theta+1)\frac{\sigma_2^2}2\right)dt-\theta\sigma_2(\BY(t))^{-\theta}dW_2(t)\\
\leq
& -2\theta\Delta_1(\BY(t))^{-\theta}-\theta\sigma_2(\BY(t))^{-\theta}dW_2(t) \text{ if } \BX(t)\geq H, \BY(t)\leq\frac{\Delta_1}L.
\end{aligned}
\end{equation}
An use of It\^o's formula shows that
\begin{equation}\label{eq-1111}
de^{2\theta\Delta_1 t}(\BY(t))^{-\theta}\leq \theta\sigma_2e^{2\theta\Delta_1 t}(\BY(t))^{-\theta}dW_2(t) \text{ if } \BX(t)\geq H, \BY(t)\leq\frac{\Delta_1}L.
\end{equation}
Let $\eta:=(n^*T)\wedge\inf\{t\geq 0: \BX(t)\leq H \text{ or } \BY(t)\geq \delta_0\}$.
It is noted that $\delta_0$ is chosen to be less than $\frac{\Delta_1}{L}$.
From \eqref{eq-1111} and an application of Dynkin's formula, we have
\begin{equation}\label{e6-prop2}
	\begin{aligned}
		\E e^{2\theta\Delta_1 (t\wedge\eta)}(\BY(t\wedge\eta))^{-\theta}&\leq y^{-\theta}, t\geq 0.
	\end{aligned}
\end{equation}
From the first line of \eqref{e3-prop2} and the fact that $F_1(u,v)u\geq 0$, we get
$$
d(\BY(t))^{-\theta}\leq \theta\left(\alpha_2+(\theta+1)\frac{\sigma_2^2}2\right)dt -\theta\sigma_2(\BY(t))^{-\theta}dW_2(t).
$$
Using arguments similar to the ones used in the process of getting \eqref{e2-thm1} from \eqref{eq-LU2} in the proof of Theorem \ref{thm1} (using appropriate stopping times until that $(\bar Y(t))^{-\theta}$ is still bounded by $n$ and then letting $n\to\infty$), yields
\begin{equation}\label{e7-prop2}
	\E_{x,y}(\BY(t))^{-\theta}\leq e^{\theta\left(\alpha_2+(\theta+1)\frac{\sigma_2^2}2\right)t}y^{-\theta},\; t\geq 0, x\geq 0, y>0.
\end{equation}

We have the following three estimates using the strong Markov property of $(\BX(t),\BY(t))$. Firstly we note that
\begin{equation}\label{e8-prop2}
	\begin{aligned}
		\E_{x,y}& \1_{\{(n^*-1)T\leq \eta\leq n^*T,\BY(\eta)
			\leq \delta_0\}} (\BY(n^*T))^{-\theta}
		\\
		\leq& \E_{x,y} \1_{\{(n^*-1)T\leq \eta\leq n^*T,\BY(\eta)< \delta_0\}}\E_{X(\eta),\BY(\eta)} (\BY(n^*T-\eta))^{-\theta}\\
		\leq &e^{\theta\left(\alpha_2+(1-\theta)\frac{\sigma_2^2}2\right)T}\E_{x,y} \1_{\{(n^*-1)T\leq \eta\leq n^*T,\BY(\eta)\leq \delta_0\}}(\BY(\eta))^{-\theta}\\
		\leq & e^{\theta\left(\alpha_2+(1-\theta)\frac{\sigma_2^2}2\right)T}e^{-2\Delta_1\theta (n^*-1)T}\E_{x,y} \1_{\{\eta\leq T,\BY(\eta)\leq \delta_0\}} e^{2\Delta_1\theta \eta}(\BY(\eta))^{-\theta}\\
		\leq &  e^{-\Delta_1\theta T}y^{-\theta}\E_{x,y} \1_{\{(n^*-1)T\leq \eta\leq n^*T,\BY(\eta)< \delta_0\}} e^{2\Delta_1\theta \eta}(\BY(\eta))^{-\theta},
	\end{aligned}
\end{equation}
where the last inequality is due to \eqref{deltan*}. Secondly, we get
\begin{equation}\label{e9-prop2}
	\begin{aligned}
		\E_{x,y}& \1_{\{\eta\leq (n^*-1)T,\BY(\eta)< \delta_0\}} (\BY(n^*T))^{-\theta}
		\\
		\leq& \E_{x,y} \1_{\{\eta\leq (n^*-1)T,\BY(\eta)\leq \delta_0\}}\E_{X(\eta),\BY(\eta)} (\BY(n^*T-\eta))^{-\theta}\\
			\leq &\E_{x,y} \1_{\{\eta\leq T,\BY(\eta)\leq \delta_0\}} (\BY(\eta))^{-\theta}\exp\{-2\Delta_1\theta(n^*T-\eta)\}\\
			\leq & e^{-2\Delta_1\theta T}\E_{x,y} \1_{\{\eta\leq T,\BY(\eta)\leq \delta_0\}} (\BY(\eta))^{-\theta}\\
			\leq &  e^{-2\Delta_1\theta T}\E_{x,y} e^{2\Delta_1\theta \eta}\1_{\{\eta\leq T,\BY(\eta)< \delta_0\}} (\BY(\eta))^{-\theta},
	\end{aligned}
\end{equation}
where in the third line we used \eqref{e4-prop2}.
Finally,
\begin{equation}\label{e10-prop2}
	\begin{aligned}
		\E_{x,y} \1_{\{\BY(\eta)\leq \delta_0\}} (\BY(n^*T))^{-\theta}
		\leq& \E_{x,y} \1_{\{\BY(\eta)\leq \delta_0\}}\E_{X(\eta),\BY(\eta)} (\BY(n^*T-\eta))^{-\theta}\\
		\leq &\E_{x,y} \delta_0^\theta e^{\theta\left(\alpha_2+(\theta+1)\frac{\sigma^2}2\right)(n^*T-\eta)}
		\\
		\leq& \hat K:=\delta_0^\theta e^{\theta\left(\alpha_2+(\theta+1)\frac{\sigma^2}2\right)n^*T}.
	\end{aligned}
\end{equation}
Adding \eqref{e8-prop2}, \eqref{e9-prop2} and \eqref{e10-prop2} side by side we have
\begin{equation}\label{e11-prop2}
\E_{x,y} (\BY(n^*T))^{-\theta}\leq e^{-\Delta_1 \theta T}\E_{x,y} e^{2\theta\Delta_1 \eta} (\BY(\eta))^{-\theta}+\hat K\leq e^{-\Delta_1 \theta T}y^{-\theta}+\hat K,
\end{equation}
where the last inequality follows from \eqref{e6-prop2}.

By the Markov property, we can recursively apply \eqref{e11-prop2}
to show that
$$
\E_{x,y} (\BY(k n^*T))^{-\theta}\leq \hat K\sum_{i=1}^k \kappa^{k-1} +\kappa^k y^{-\theta}\leq \frac{\hat K}{1-\kappa} + \kappa^k y^{-\theta}, \text{ where } \kappa:=e^{-\Delta_1\theta T}<1.
$$
This and \eqref{e7-prop2} imply that
$$\E_{x,y} (\BY(t))^{-\theta}\leq e^{\theta(\alpha_2+\frac{\sigma^2}2(\theta+1)n^*T}\left(\frac{\hat K}{1-\kappa} + \kappa^k y^{-\theta}\right)\;\forall t\in[kn^*T, (k+1)n^*T],$$
which is equivalent to \eqref{e0-prop2}.
\end{proof}

The rest of this section is devoted to proving Theorem \ref{thm3}.
Before constructing suitable coupling systems, we need the following bound for the growth rate of the solution on the boundary corresponding to $z=0$.
\begin{lm}\label{lm1}
For any $\eps\in(0,1)$, $\delta>0$, there exists $M_0=M_0(\eps, \delta, x,y)$ such that
$$\PP_{x,y}\left\{\BX(t)+\BY(t)+(\BX(t))^{-1}+(\BY(t))^{-1}\leq M_0e^{\delta t},\;\forall t\geq 0\right\}\geq 1-\eps.
$$
\end{lm}
\begin{proof}
		Pick $\theta>0$ satisfying \eqref{e0-prop2} and
let $\bar V(x,y)=x+y+y^{-\theta}$.
In view of \eqref{e1-thm1} and \eqref{e0-prop2},
we have
\begin{equation}\label{eq-barC}
\E_{x,y}\bar V(\BX(t), \BY(t))\leq \bar C_{x,y},\;\forall t\geq 0,
\end{equation}
for some constant $\bar C_{x,y}$ independent of $t$.
It\^o's formula yields
\begin{equation}\label{eq-barV}
\begin{aligned}
d\bar V(\BX(t),\BY(t))=&\Lambda-\alpha_1\BX(t)-(\alpha_2-\alpha_4)\BY(t)-\theta (\BY(t))^{-\theta}\left(F_1(\BX(t),\BY(t)\BX(t)-\alpha_2-(\theta+1)\frac{\sigma_2^2}2\right)dt\\
&+\sigma_1\BX(t)dW_1(t)+\sigma_2\BY(t)dW_2(t)-\theta\sigma_2\BY^{-\theta}dW_2(t)\\
\leq& A_0\bar V(\BX(t),\BY(t))dt+\sigma_1\BX(t)dW_1(t)+\sigma_2\BY(t)dW_2(t)-\theta\sigma_2(\BY(t))^{-\theta}dW_2(t)
\end{aligned}
\end{equation}
for some $A_0>0$.
For any $c>0$, let $\bar\tau_{c}:=\inf\{t\geq0: \bar V(\BX(t),\BY(t))\geq c\}.$
Equation \eqref{eq-barV} together with an application of Dynkin's formula implies that
$$\E_{x,y}e^{-A_0(\bar\tau_{c}\wedge t)} \bar V(\BX(\bar\tau_{c}\wedge t),\BY(\bar\tau_{c}\wedge t))\leq \bar V(x,y),\quad\forall t\geq0.$$
As a result
$$\E_{x,y} \bar V(\BX(\bar\tau_{c}\wedge t),\BY(\bar\tau_{c}\wedge t))\leq \bar V(x,y)e^{A_0t},\quad\forall t\geq0.$$
Therefore, for any $c>0$, applying Markov's inequality we have
\begin{equation}\label{eq-supV}
\PP\left\{\sup_{t\in[0,1]}\bar V(\bar X(t),\bar Y(t))\geq c\right\}\leq \frac1{c}\E_{x,y} \bar V(\BX(\bar\tau_{c}\wedge 1),\BY(\bar\tau_{c}\wedge 1))\leq \frac{e^{A_0}}{c}\bar V(x,y).
\end{equation}
For $\eps>0$, $\delta>0$, pick $ M_0$ sufficiently large such that $\frac{e^{A_0}\bar C_{x,y}}{M_0}\sum_{n=1}^\infty e^{-\delta\theta n}<\eps$.
By the Markov property of $(\BX,\BY)$, \eqref{eq-barC}, and \eqref{eq-supV}, we have
$$
\PP\left\{\sup_{t\in[n,n+1]}\bar V(\bar X(t),\bar Y(t))> M_0 e^{\delta n}\right\}\leq \frac{e^{A_0}}{M_0 e^{\delta n}}\E_{x,y}\bar V(\bar X(n),\bar Y(n))\leq \frac{e^{A_0}\bar C_{x,y}}{M_0 e^{\delta n}};
$$
which leads to
\begin{equation}\label{e5-lm1}
\PP\left\{\sup_{t\in[n,n+1]}\bar V(\bar X(t),\bar Y(t))\leq M_0 e^{\delta\theta n}, \text{ for all } n\in\N\right\}>1-\sum_{n=1}^\infty\frac{\bar C_{x,y}}{ M_0 e^{\delta\theta n}}.
\end{equation}
From \eqref{e5-lm1} and the definition of $M_0$, we obtain the desired result.

Next, we need to bound $X^{-1}(t)$.
Using the variation of constants formula
(see \cite[Chapter 3]{MAO97}), we can write $\BX(t)$ in the form
\begin{equation}
	\BX(t)=\Phi^{-1}(t)\left[\int_0^t\Phi(s)\left(\Lambda-F_1(\BX(s), \BY(s))\BY(s)+\alpha_4\BY(s)\right)ds\right],
\end{equation}
where
$$
\Phi(t):=\exp\left\{\left(\alpha_1+\frac{\sigma_1^2}2\right)t-\sigma_1 W_1(t)\right\}.
$$
In view of \eqref{e5-lm1}, for any $\eps>0$, there exists $M_2=M_2(\eps,\delta,x,y)>0$ such that
\begin{equation}\label{e8-lm1}
	\PP_{x,y}\left\{ \BY(t)\leq M_2e^{\delta\theta t},\,\forall\, t\geq0\right\}\geq 1-\frac\eps2.
\end{equation}
It is easily seen that there is $M_3=M_3(\eps,\delta)>0$ such that
$$
\PP\left\{\sigma_1|W_1(t)|\leq M_3e^{\delta\theta t}, \;\forall t\geq 0\right\}\geq 1-\frac\eps2.
$$

On the other hand, given that
$\sigma_1|W_1(t)|\leq M_3e^{\delta\theta t}$ and
$\Lambda-F_1(\BX(t), \BY(t)\BY(t)+\alpha_4\BY(t)\leq \Lambda+\alpha_4\BY(t)\leq \alpha_4 M_2e^{\delta\theta t}+\Lambda$,
	one can see from \eqref{e8-lm1} that
$$\BX(t)\geq \frac{e^{-2\delta\theta t}}{M_4}\geq \frac{e^{-2\delta t}}{M_4} \text{ for some constant } M_4 \text{ depending on } M_2, M_3.$$
Combining this with \eqref{e5-lm1} concludes the proof (after re-assigning $\delta:=\delta\theta$).
\end{proof}

	Since $F_1(u,v)u$ and $F_2(v,w)v$ are Lipschitz and $F_1$ and $F_2$ are bounded,
there exists $c_0>0$ such that
\begin{equation}\label{boundc0}
\begin{aligned}
(&u_1-u_2)[(\Lambda-F_1(u_1,v_1)u_1v_1-\alpha_1u_1+\alpha_4 v_1)-(\Lambda-F_1(u_2,v_2)u_2v_2-\alpha_1u_2+\alpha_4 v_2)]\\
	&+(v_1-v_2)[F_1(u_1,v_1)u_1v_1-\alpha_2v_1-(F_1(u_2,v_2)u_2v_2 -F_2(v_2,w)v_2w-\alpha_2v_2]
	\\
	&+\sigma_1^2(u_1-u_2)^2+\sigma_2^2(v_1-v_2)^2\\
	\leq&  \frac12\left(c_0(1+u_1+v_1+u_2+v_2)^2[(u_1-u_2)^2+(v_1-v_2)^2]+ c_0 w_2^2\right),\quad\forall u_1,u_2,v_1,v_2,w\geq 0.
\end{aligned}
\end{equation}
Let
\begin{equation}\label{eq-delta-N}
\gamma_0:=-\frac{\lambda_2}3>0,
\text{ and } \wdt N> \gamma_0 +(\sigma_1^2\vee\sigma_2^2) +c_0,
\end{equation}
and consider
the coupling system:


\begin{equation}\label{coup1}
	\begin{cases}
		d\BX(t)=& [\Lambda- F_1(\BX(t),\BY(t))\BX(t)\BY(t)-\alpha_1\BX(t)+\alpha_4 \BY(t)]dt+\sigma_1 \BX(t)dW_1(t)\\
		d\BY(t)=& [F_1(\BX(t),\BY(t))\BX(t)\BY(t)-\alpha_2\BY(t)]dt+\sigma_2 \BY(t)dW_2(t)\\
		d\wdt X(t)=& [\Lambda- F_1(\wdt X(t),\wdt Y(t))\wdt X(t)\wdt Y(t)-\alpha_1\wdt X(t)+\alpha_4\wdt  Y(t)+\alpha_5\wdt  Z(t)]dt+\sigma_1\wdt  X(t)dW_1(t)\\
		&-\wdt N(1+\BX(t)+\wdt X(t)+\BY(t)+\wdt Y(t))^2(\BX(t)-\wdt X(t))dt\\
		d\wdt Y(t)=& [F_1(\wdt X(t),\wdt Y(t))\wdt X(t)\wdt Y(t)-F_2(\wdt Y(t),\wdt Z(t))\wdt Y(t)\wdt Z(t)-\alpha_2\wdt Y(t)]dt+\sigma_2\wdt  Y(t)dW_2(t)\\
		&-\wdt N(1+\BX(t)+\wdt X(t)+\BY(t)+\wdt Y(t))^2(\BY(t)-\wdt Y(t))dt\\
		d\wdt Z(t)=& [F_2(\wdt Y(t),\wdt Z(t))\wdt Y(t)\wdt Z(t)-\alpha_3\wdt Z(t)]dt+\sigma_3\wdt  Z(t)dW_3(t).\\
	\end{cases}
\end{equation}
\begin{rem}
Because the methods in existing work (such as those from \cite{HN16}) do not work, this coupled system is introduced to compare the solution near the boundary (when $Z(t)$ is small) and the solution on the boundary (when $Z(t)=0$). Based on \eqref{boundc0}, the term $-\wdt N(1+\BX(t)+\wdt X(t)+\BY(t)+\wdt Y(t))^2(\BX(t)-\wdt X(t))$ and $-\wdt N(1+\BX(t)+\wdt X(t)+\BY(t)+\wdt Y(t))^2(\BY(t)-\wdt Y(t))$ on the coupled equations of $d\wdt X(t)$ and $d\wdt Y(t)$ in  \eqref{coup1} respectively are needed to make sure that $(\wdt X(t),\wdt Y(t))$ will approach $(\BX(t),\BY(t))$ with a large probability. We note that although the comparison in a finite interval is standard, one cannot use it to obtain the desired result which requires the two solutions to be close with a large probability in the infinite interval $[0,\infty)$.
\end{rem}
The next proposition will quantify how close $(\bar X(t),\bar Y(t))$ and $(\wdt X(t),\wdt Y(t))$ are when $\bar Z(t)$ is small.
\begin{prop}\label{prop2.1}
	For $\delta>0$,
let
	$$\wdt\tau_\delta:=\inf\left\{t\geq0: \wdt Z(t)\geq\delta e^{-\gamma_0 t}\right\}.$$
	There is a constant $\wdt C$ independent of $|\bar x-\wdt x|$, $|\bar y-\wdt y|$ and $\delta$ such that
	\begin{equation}\label{e1-prop2.1}
		\E \sup_{0\leq t\leq\wdt\tau_\delta}e^{2\gamma_0 t}[(\BX(t)-\wdt X(t))^2+(\BY(t)-\wdt Y(t))^2] \leq \wdt C((\bar x-\wdt x)^2+(\bar y-\wdt y)^2+\delta^2).
	\end{equation}
Moreover, there are $\wdt M_{\eps, x,y},\wdt m_{\eps,x,y}>0$ (depending only on $\eps,x,y$) such that
	\begin{equation}\label{e2-prop2.1}
\PP_{x,y,\wdt\bs}\left\{\int_0^{\wdt\tau_\delta}(|v_1(t)|^2+|v_2(t)|^2)dt\geq \wdt M_{\eps, x,y}((\bar x-\wdt x)^2+(\bar y-\wdt y)^2+\delta^2) \right\}\leq \eps,
	\end{equation}
as long as $(\bar x-\wdt x)^2+(\bar y-\wdt y)^2+\delta^2\leq \wdt m_{\eps,x,y}$,
where
$$\textstyle v_1(t)=\frac{\wdt N(1+\BX(t)+\wdt X(t)+\BY(t)+\wdt Y(t))^2(\BX(t)-\wdt X(t))}{\sigma_1\wdt X(t)},$$
and
$$\textstyle v_2(t)=\frac{\wdt N(1+\BX(t)+\wdt X(t)+\BY(t)+\wdt Y(t))^2(\BY(t)-\wdt Y(t))}{\sigma_1\wdt Y(t)}.$$
\end{prop}
\begin{proof}
Applying It\^o's formula to \eqref{coup1} and using \eqref{boundc0}, we have
\begin{equation}\label{e01-prop2.1}
	\begin{aligned}
	d[&(\BX(t)-\wdt X(t))^2+(\BY(t)-\wdt Y(t))^2]\\ \leq&\left(-(2\wdt N-c_0)(1+\BX(t)+\wdt X(t)+\BY(t)+\wdt Y(t))((\BX(t)-\wdt X(t))^2+(\BY(t)-\wdt Y(t))^2)\right)dt\\
	&+c_0|\wdt Z(t)|^2dt+2\sigma_1(\BX(t)-\wdt X(t))^2dW_1(t)+2\sigma_2(\BY(t)-\wdt Y(t))^2dW_2(t).
	\end{aligned}
\end{equation}
Here and thereafter, $C$ is a generic constant, whose value can be different in different lines, but which is independent of $|\bar x-\wdt x|, |\bar y-\wdt y|$ and $\delta$.
By It\^o's formula and Cauchy's inequality we have from \eqref{e01-prop2.1} that
\begin{equation*}
	\begin{aligned}
d&e^{4\gamma_0 t}[(\BX(t)-\wdt X(t))^2+(\BY(t)-\wdt Y(t))^2]^2\\
\leq& -\big(4\wdt N-4\gamma_0-4(\sigma_1^2\vee\sigma_2^2)-4c_0\big)e^{4\gamma_0 t}[(\BX(t)-\wdt X(t))^2+(\BY(t)-\wdt Y(t))^2]+ 2c_0e^{4\gamma_0 t}|\wdt Z(t)|^4dt\\
&+ 2e^{4\gamma_0 t}[(\BX(t)-\wdt X(t))^2+(\BY(t)-\wdt Y(t))^2]^2\left(\sigma_1(\BX(t)-\wdt X(t))^2dW_1(t)+\sigma_2(\BY(t)-\wdt Y(t))^2dW_2(t)\right).
	\end{aligned}
\end{equation*}
Then, by introducing suitable stopping times and passing to the limit, as was done in the process of getting  \eqref{e2-thm1} from \eqref{eq-LU2} in the proof of Theorem \ref{thm1}, one can obtain
\begin{align*}
\Big(4\wdt N&-4\gamma_0-4(\sigma_1^2\vee\sigma_2^2)-4c_0\Big)\E \int_0^{t\wedge\wdt\tau_\delta} e^{4\gamma_0 s}[(\BX(s)-\wdt X(s))^2+(\BY(s)-\wdt Y(s))^2]^2\\
\leq&
2\E \int_0^{t\wedge\wdt\tau_\delta}c_0e^{4\gamma_0 s}|\wdt Z(s)|^4ds
+ ((\bar x-\wdt x)^2+(\bar y-\wdt y)^2)^2.
\end{align*}
This leads to
\begin{equation}\label{e3-prop2.1}
\E \int_0^{t\wedge\wdt\tau_\delta} e^{4\gamma_0 s}[(\BX(s)-\wdt X(s))^2+(\BY(s)-\wdt Y(s))^2]^2\leq C((\bar x-\wdt x)^4+(\bar y-\wdt y)^4+\delta^4).\end{equation}
Moreover, we have from \eqref{e01-prop2.1} and It\^o's formula that
\begin{equation*}
	\begin{aligned}
		d&e^{2\gamma_0 t}[(\BX(t)-\wdt X(t))^2+(\BY(t)-\wdt Y(t))^2]\\&\leq -(2\lambda-2\gamma_0-c_0)e^{\gamma_0 t}[(\BX(t)-\wdt X(t))^2+(\BY(t)-\wdt Y(t))^2]+ c_0e^{2\gamma_0 t}|\wdt Z(t)|^2dt\\
		&\qquad+ e^{2\gamma_0 t}\left(\sigma_1(\BX(t)-\wdt X(t))^2dW_1(t)+\sigma_2(\BY(t)-\wdt Y(t))^2dW_2(t)\right).
	\end{aligned}
\end{equation*}
From this, we get that
\begin{equation}\label{e4-prop2.1}
\begin{aligned}
\E& \sup_{t\leq T\wedge\wdt\tau_\delta} e^{2\gamma_0 t}[(\BX(t)-\wdt X(t))^2+(\BY(t)-\wdt Y(t))^2]\\
\leq &\E \int_0^{T\wedge\wdt\tau_\delta} c_0e^{2\gamma_0 s}|\wdt Z(s))|^2ds\\
&+\E \sup_{t\leq T\wedge\wdt\tau_\delta}\int_0^t\left\{e^{2\gamma_0 s}\left(\sigma_1(\BX(s)-\wdt X(s))^2dW_1(s)+\sigma_2(\BY(s)-\wdt Y(s))^2dW_2(s)\right)\right\}.
\end{aligned}
\end{equation}
In view of the Burkholder-Davis-Gundy inequality, we have
\begin{equation}\label{e5-prop2.1}
	\begin{aligned}
		\E \sup_{t\leq T\wedge\wdt\tau_\delta}&\int_0^t\left\{e^{2\gamma_0 s}\left(\sigma_1(\BX(s)-\wdt X(s))^2dW_1(s)+\sigma_2(\BY(s)-\wdt Y(s))^2dW_2(s)\right)\right\}\\
		\leq& C\left[\E \int_0^{t\wedge\wdt\tau_\delta} e^{4\gamma_0 s}[(\BX(s)-\wdt X(s))^2+(\BY(s)-\wdt Y(s))^2]^2\right]^{\frac12}\\
		\leq& C((\bar x-\wdt x)^2+(\bar y-\wdt y)^2+\delta^2),
	\end{aligned}
\end{equation}
where in the last line we used \eqref{e3-prop2.1}.
In addition, since $|Z(s)|\leq \delta e^{-\gamma_0 s}$ for any $s\leq\wdt\tau_\delta$, it can be seen that
\begin{equation}\label{e6-prop2.1}
\E \int_0^{T\wedge\wdt\tau_\delta} c_0e^{2\gamma_0 s}|\wdt Z(s))|^2ds\leq C\delta^2.
\end{equation}
Using \eqref{e5-prop2.1} and \eqref{e6-prop2.1} in \eqref{e4-prop2.1}, we obtain \eqref{e1-prop2.1}.

We next prove \eqref{e2-prop2.1}.
In view of Lemma \ref{lm1}, there is $M_{\eps,x,y}$ such that
\begin{equation}\label{e9-prop2.1}
\PP_{x,y}\bigg(\wdt\Omega_3:=\left\{[1+\BX(t)+\BX(t)^{-1}+\BY(t)+\BY^{-1}(t)]\leq  M_{\eps,x,y},\;\forall t\geq 0\right\}\bigg)\geq 1-\frac{\eps}2.
\end{equation}
By virtue of \eqref{e1-prop2.1}, there is $\wdt C_0$ independent of $(\bar x-\wdt x)^2+(\bar y-\wdt y)^2+\delta^2$ such that
\begin{equation}\label{e10-prop2.1}
\begin{aligned}
	\PP_{x,y,\wdt\bs}&\left(\wdt\Omega_4:=\Big\{e^{2\gamma_0 t}[(\BX(t)-\wdt X(t))^2+(\BY(t)-\wdt Y(t))^2] \leq \frac{\wdt C_0((\bar x-\wdt x)^2+(\bar y-\wdt y)^2+\delta^2)}\eps,\;\forall 0\leq t\leq \wdt\tau_\delta \Big\}\right)\\
	&\geq 1-\frac{\eps}2.
	\end{aligned}
\end{equation}
For $t\leq\tau_\delta$, if $\BX(t)\geq  M_{\eps,x,y}^{-1} e^{-\gamma_0 t/4}$ and
$(\BX(t)-\wdt X(t))\leq \frac12M_{\eps,x,y}^{-1} e^{-\gamma_0 t/4}$, then we have
\begin{equation}\label{e6a-prop2.1}
	\frac1{\wdt X(t)}\leq \frac1{\BX(t)+(\BX(t)-\wdt X(t))}\leq \frac{1}{{M_{\eps,x,y}}^{-1} e^{-\gamma_0 t/4}+(\BX(t)-\wdt X(t))}\leq 2 M_{\eps,x,y} e^{\gamma_0 t/4}.
\end{equation}
Likewise,
\begin{equation}\label{e6b-prop2.1}
	\frac1{\wdt Y(t)}\leq 2 M_{\eps,x,y} e^{\gamma_0 t/4}
	\text{ if provided }\BY(t)\geq M_{\eps,x,y}^{-1} e^{-\gamma_0 t/4} \text{ and
	} (\BY(t)-\wdt Y(t))\leq \frac12 M_{\eps,x,y}^{-1} e^{-\gamma_0 t/4}.
\end{equation}
Observe that if $(\bar x-\wdt x)^2+(\bar y-\wdt y)^2+\delta^2\leq \frac{\eps}{4\wdt C_0M_{\eps,x,y}^2}$
then for all $\omega\in\wdt\Omega_3$,
$$(\BX(t)-\wdt X(t))\vee (\BY(t)-\wdt Y(t))\leq \left(\frac{\wdt C_0((\bar x-\wdt x)^2+(\bar y-\wdt y)^2+\delta^2)}\eps e^{-2\gamma_0 t}\right)^{-\frac12}\leq \frac12M_{\eps,x,y}^{-1} e^{-\gamma_0 t/4}.$$
This together with \eqref{e6a-prop2.1} and \eqref{e6b-prop2.1} implies that for all $\omega\in\wdt\Omega_3\cap\wdt\Omega_4$,
\begin{equation}\label{e6c-prop2.1}
	\frac1{\wdt X(t)}\vee	\frac1{\wdt Y(t)}\leq 2M_{\eps,x,y,1}e^{\gamma_0t/4}\text{ provided that }(\bar x-\wdt x)^2+(\bar y-\wdt y)^2+\delta^2\leq \frac{\eps}{2\wdt C_0M_{\eps,x,y}}.
\end{equation}
Note that
\begin{equation}\label{eq-v1v2}
|v_1(t)|^2+|v_2(t)|^2\leq \frac{4\lambda}{\sigma_1^2\vee\sigma_2^2}\left(\wdt X^{-2}(t)\wedge\wdt Y^{-2}(t)\right)[3+\BX(t)+\BY(t)]^4\left((\BX(t)-\wdt X(t))+(\BY(t)-\wdt Y(t))\right)^2.
\end{equation}
Combining \eqref{e9-prop2.1}, \eqref{e10-prop2.1}, \eqref{e6c-prop2.1}, and \eqref{eq-v1v2}, we have, when  $(\bar x-\wdt x)^2+(\bar y-\wdt y)^2+\delta^2\leq \frac{\eps}{2\wdt C_0M_{\eps,x,y}}$, that
$$
\PP\left\{|v_1(t)|^2+|v_2(t)|^2\leq M'_{\eps,x,y}\frac{\wdt C_0((\bar x-\wdt x)^2+(\bar y-\wdt y)^2+\delta^2)}\eps e^{-\gamma_0 t/2}\text{ for all } 0\leq t\leq \wdt\tau_\delta\right\}\geq 1-\eps,
$$
for some $M'_{\eps,x,y}$. This implies \eqref{e2-prop2.1}.
The proof is complete.
\end{proof}
In the next lemma, we will show that $Z(t)$ converges to $0$ (exponentially fast) whenever the solution starts in a neighborhood of the boundary corresponding to $z=0$.
\begin{lm}\label{lm2}
	For any $(x,y)\in\R^{2,\circ}_+$ and $\eps\in(0,1)$, there exists $\varsigma=\varsigma(x,y,\eps)$ such that
	$$
	\PP_{\wdt\bs}\left\{\lim_{t\to\infty}\frac{\ln Z(t)}t=\lambda_2<0\right\}>1-\eps,
	$$
	for all $\wdt\bs=(\wdt x,\wdt y,\wdt z)$ satisfying $(\wdt x-x)^2+(\wdt y-y)^2+\wdt z^2\leq\varsigma^2$.
\end{lm}
\begin{proof}
	First, we choose $\delta=\delta(\eps,x,y)>0$ such that
	\begin{equation}\label{wdtM}
	2\wdt M_{\eps, x,y}\delta^2\leq\eps \text{
	and }2\eps^2\wdt M_{\eps, x,y}2\delta\leq\eps,
\end{equation}
	where $\wdt M_{\eps,x,y}$ is determined as in \eqref{e2-prop2.1}.
Define
$$
\Omega_1:=\left\{e^{2\gamma_0 t}[(\BX(t)-\wdt X(t))^2+(\BY(t)-\wdt Y(t))^2] \leq \frac{\wdt C_0 ((\bar x-\wdt x)^2+(\bar y-\wdt y)^2+\delta^2)}\eps\right\}.
$$
Because of the egodicity, we have
$$
\PP_{x,y}\left\{\frac1t\int_0^t F_2(\bar Y(t),0)\bar Y(t) dt=\lambda_2+\alpha_3+\frac{\sigma_3^2}2\right\}=1.$$
Therefore, we can find $T>0$ such that $\PP_{x,y}(\Omega_2)>1-\eps$ where
$$
\Omega_2:=\left\{\frac1t\int_0^t F_2(\bar Y(t),0)\bar Y(t) dt-\alpha_3-\frac{\sigma_3^2}2\leq \lambda_2+\gamma_0,\;\forall t\geq T\right\}.
$$
In view of \eqref{e3-thm1}, we can find $\wdt D_{x,y,\eps, T}>0$ such that $\PP_{x,y}(\Omega_3)\geq 1-\eps$ where
$$
\Omega_3:=\left\{\int_0^t F_2(\bar Y(t),0)\bar Y(t) ds\leq \wdt D_{x,y,\eps, T},\;\forall t\leq T\right\}.
$$
By the exponential martingale inequality, see e.g. \cite{MAO97}, we have $\PP(\Omega_4)\geq 1-\eps$ where
$$
\Omega_4:=\left\{\sigma_3W(t)\leq \frac2{\gamma_0}|\ln\eps|+\gamma_0 t,\; \forall\, t\geq0\right\}.
$$
For $0\leq t\leq T\wedge\wdt\tau_\delta$, $\omega\in\cap_{i=1}^4\Omega_i$, we have
\begin{equation}\label{e11-prop2.1}
	\begin{aligned}
		\ln\wdt Z(t)=&\ln\wdt z+\int_0^tF_2(\wdt Y(s),\wdt Z(s))\wdt Y(s)ds-\left(\alpha_3-\frac{\sigma_3^2}2\right)t+\sigma_3W(t)\\
	\leq& \ln\wdt z+\int_0^tF_2( \BY(s),0)\BY(s)ds-\left(\alpha_3-\frac{\sigma_3^2}2\right)t+\sigma_3W(t)+L\int_0^t |(\wdt Z(t))^2+(\BY(s)-\wdt Y(s))^2|^{\frac12}ds\\
	\leq& \ln\wdt z +\frac2{\gamma_0}|\ln\eps| +\wdt D_{x,y,\eps, T}+L\frac{\sigma^2}{4\gamma_0}+L\int_0^t e^{-2\gamma_0s}\frac{\wdt C_{0} ((\bar x-\wdt x)^2+(\bar y-\wdt y)^2+\delta^2)}\eps ds\\
	\leq&\ln\wdt z +\frac2{\gamma_0}|\ln\eps| +\wdt D_{x,y,\eps, T}+L\frac{\sigma^2}{4\gamma_0}+\frac{L\wdt C_{0} ((\bar x-\wdt x)^2+(\bar y-\wdt y)^2+\delta^2)}{2\eps\gamma_0}.
	\end{aligned}
\end{equation}
If $\ln\wdt z< \ln\varsigma:=\ln\delta - \left(\frac2{\gamma_0}|\ln\eps| +\wdt D_{x,y,\eps, T}+L\frac{\sigma^2}{4\gamma_0}+\frac{2L\wdt C_{0} ((\bar x-\wdt x)^2+(\bar y-\wdt y)^2+\delta^2)}{\eps\gamma_0} \right)$
then it is easily seen that
$\wdt\tau_\delta\geq T$ for any $\omega\in\cap_{i=1}^4\Omega_i$ because $\ln\wdt Z(t)\leq \ln\delta$ for any $t\leq T\wedge\wdt\tau_\delta$ and $\omega\in\cap_{i=1}^4\Omega_i$.

For $T\leq t\leq \wdt\tau_\delta$ we have
$$
\ln \wdt Z(t)\leq \ln\wdt z+ (\lambda_2+4\gamma_0)t+\frac2{\gamma_0}|\ln\eps|+L\frac{\sigma^2}{4\gamma_0}+\frac{L\wdt C_{0} ((\bar x-\wdt x)^2+(\bar y-\wdt y)^2+\delta^2)}{2\eps\gamma_0} <\ln\delta.$$
Thus, we must have $\wdt\tau_\delta=\infty$ for $\omega\in\cap_{i=1}^4\Omega_i$
and that
$\limsup \frac{\ln\wdt Z(t)}t\leq \lambda_2-4\gamma_0<0$ for $\omega\in\cap_{i=1}^4\Omega_i$.

For the rest of this proof, we always assume that $(\wdt x-x)^2+(\wdt y-y)^2+\wdt z^2\leq\varsigma^2<\wdt m_{\eps,x,y}^2$, where $\wdt m_{\eps,x,y}$ is chosen as in \eqref{e2-prop2.1}. Consider the following coupled system:
	\begin{equation}\label{coup2}
		\begin{cases}
			d\BX(t)=& [\Lambda- F_1(\BX(t),\BY(t))\BX(t)\BY(t)-\alpha_1\BX(t)+\alpha_4 \BY(t)]dt+\sigma_1 \BX(t)dW_1(t)\\
			d\BY(t)=& [F_1(\BX(t),\BY(t))\BX(t)\BY(t)-\alpha_2\BY(t)]dt+\sigma_2 \BY(t)dW_2(t)\\
			d\hat X(t)=& [\Lambda- F_1(\hat X(t),\hat Y(t))\hat X(t)\hat Y(t)-\alpha_1\hat X(t)+\alpha_4\hat  Y(t)+\alpha_5\hat  Z(t)]dt+\sigma_1\hat  X(t)dW_1(t)\\
			&-\wdt N\1_{\{t<\wdt\tau_\delta\}}(1+\BX(t)+\hat X(t)+\BY(t)+\hat Y(t))^2(\BY(t)-\hat Y(t))dt\\
			d\hat Y(t)=& [F_1(\hat X(t),\hat Y(t))\hat X(t)\hat Y(t)-F_2(\hat Y(t),\hat Z(t))\hat Y(t)\hat Z(t)-\alpha_2\hat Y(t)]dt+\sigma_2\hat  Y(t)dW_2(t)\\
			&-\wdt N\1_{\{t<\wdt\tau_\delta\}}(1+\BX(t)+\hat X(t)+\BY(t)+\hat Y(t))^2(\BY(t)-\hat Y(t))dt\\
			d\hat Z(t)=& [F_2(\hat Y(t),\hat Z(t))\hat Y(t)\hat Z(t)-\alpha_3\hat Z(t)]dt+\sigma_2\hat  Y(t)dW_2(t).
		\end{cases}
	\end{equation}	
Then, $(\hat X(t),\hat Y(t),\hat Z(t))\equiv (\wdt X(t),\wdt Y(t),\wdt Z(t))$ up to $\wdt\tau_\delta$.
Moreover, let $\Q_{x,y,\wdt\bs}$ be the measure defined by
$$
\dfrac{d\Q_{x,y,\wdt\bs}}{d\PP_{x,y,\wdt\bs}}=\exp\left\{-\int_0^{\wdt\tau_\delta}[v_1(s)dW_1(s)+v_2(s)dW_2(s)]-\int_0^{\wdt\tau_\delta}[v_1^2(s)+v_2^2(s)]ds\right\}.
$$
Then,
$\left(W_1(t)+\int_0^{t\wedge\wdt\tau_\delta}v_1(s)ds,W_2(t)+\int_0^{t\wedge\wdt\tau_\delta}v_2(s)ds\right)$ is a standard two-dimensional Brownian motion under $\Q$.
As a result,
$(\hat X(t),\hat Y(t),\hat Z(t))$ is the solution to \eqref{main} with initial condition $\wdt\bs$ under $\Q$.

Let
$$
\Omega_5:=\left\{\int_0^{\wdt\tau_\delta}|v_1(t)|^2+|v_2(t)|^2)dt\geq \wdt M_{\eps,x,y}\right\},$$
and
$$
\Omega_6:=\left\{\int_0^t (v_1(s)dW_1(s)+v_2(s)dW_2(s))\leq\frac{\eps^2}{2\delta} \int_0^{t}|v_1(s)|^2+|v_2(s)|^2)ds+\eps\right\}.
$$
In view of the exponential martingale inequality (see e.g. \cite{MAO97}), if $\delta\leq \eps^3/(-\ln\eps)$ we have
$$
\PP_{x,y,\wdt\bs}(\Omega_6)\geq 1-e^{\eps^3/\delta}\geq 1-\eps.
$$
For $\omega\in\Omega_5\cap\Omega_6$, we have
\begin{equation}
	\begin{aligned}
		\dfrac{d\Q_{x,y,\wdt\bs}}{d\PP_{x,y,\wdt\bs}}=&\exp\left\{-\int_0^{\wdt\tau_\delta}[v_1(s)dW_1(s)+v_2(s)dW_2(s)]-\int_0^{\wdt\tau_\delta}[v_1^2(s)+v_2^2(s)]ds\right\}\\
		\geq & \exp\left\{-\frac{\eps^2}{2\delta} \int_0^{t}|v_1(s)|^2+|v_2(s)|^2)ds-\eps-\int_0^{\wdt\tau_\delta}[v_1^2(s)+v_2^2(s)]ds\right\}\\
		\geq & e^{-\frac{\eps^2 \wdt M_{\eps,x,y}2\delta^2}{2\delta}-\eps -\wdt M_{\eps,x,y}2\delta^2}\geq e^{-3\eps}\geq 1-4\eps \text{ (due to \eqref{wdtM})}.
	\end{aligned}
\end{equation}
Thus,
$$\Q_{x,y,\wdt\bs}(\cap_{i=1}^6\Omega_i)=\int_{\cap_{i=1}^6\Omega_i} \dfrac{d\Q_{x,y,\wdt\bs}}{d\PP_{x,y,\wdt\bs}} d\PP_{x,y,\wdt\bs}\geq (1-4\eps)\PP_{x,y,\wdt\bs}(\cap_{i=1}^6\Omega_i)\geq (1-4\eps)(1-6\eps)>1-10\eps.
$$

Note that, for $\omega\in\cap_{i=1}^6\Omega_i$, $\wdt\tau_\delta=\infty$ and
\begin{equation}\label{e8-lm2}
	\limsup_{t\to\infty}\frac{\ln\hat Z(t)}t=\limsup_{t\to\infty}\frac{\ln\hat Z(t)}t\leq \lambda_2-4\gamma_0<0.
\end{equation}
Because
\begin{equation}\label{e7-lm2}e^{\gamma_0 t}[(\BX(t)-\wdt X(t))^2+(\BY(t)-\wdt Y(t))^2] \leq \frac{\wdt C_{0} ((\bar x-\wdt x)^2+(\bar y-\wdt y)^2+\delta^2)}\eps,
\end{equation}
and the random occupation measure
$$\bar\Pi_t:=\frac1t\int_0^t \1_{\{(\bar X(s),\bar Y(s))\in\cdot\}}ds$$ converges weakly to $\mu_{12}$ (as a measure on $(0,\infty)^2$) as $t\to\infty$ almost surely,
we can claim that
$$\hat\Pi_t(\cdot):=\frac1t\int_0^t \1_{\{(\hat X(s),\hat Y(s),\hat Z(s))\in\cdot\}}ds$$
converges weakly to $\mu_{12}$ (as a measure on $(0,\infty)^2\times\{0\}$) as $t\to\infty$ for almost all $\omega\in\cap_{i=1}^6\Omega_i.$
We also deduce from \eqref{e8-lm2} and \eqref{e7-lm2} and the Lipschitz continuity of $F_2$ that
\begin{equation}\label{e9-lm2}
	\begin{aligned}
	\lim_{t\to\infty}\frac{\ln \hat Z(t)}t=&\lim_{t\to\infty} \frac1t\int_0^t F_2(\hat Y(s),\hat Z(s))ds-\alpha_3-\frac{\sigma^2_3}2+\lim_{t\to\infty}\frac{\sigma_3 W_3(t)}t\\
	=&
\lim_{t\to\infty} \frac1t\int_0^t F_2(\BY(s),0)ds-\alpha_3-\frac{\sigma^2_3}2+\lim_{t\to\infty}\frac{\sigma_3 W_3(t)}t\\
=& \lambda_2<0, \text{ for almost all } \omega\in\cap_{i=1}^6\Omega_i.
\end{aligned}
\end{equation}
Finally, because $(\hat X(t),\hat Y(t),\hat Z(t))$ is the solution to \eqref{main} with initial condition $\hat\bs$ under $\Q$ and $\Q_{x,y,\wdt\bs}(\cap_{i=1}^6\Omega_i)\geq 1-10\eps$, we can claim that
$$
\PP_{\wdt\bs}\left\{\lim_{t\to\infty}\frac{\ln Z(t)}t=\lambda_2<0\right\}=\Q_{x,y,\wdt\bs}\left\{\lim_{t\to\infty}\frac{\ln \hat Z(t)}t=\lambda_2<0\right\}\geq 1-10\eps,
$$
as long as $(\wdt x-x)^2+(\wdt y-y)^2+\wdt z^2\leq\varsigma^2$.
The proof is complete.
\end{proof}
We are ready to prove Theorem \ref{thm3}.
\begin{proof}[Proof of Theorem \ref{thm3}.]
	Lemma \ref{lm2} implies that there is no invariant measure on $\R^{3,\circ}_+$.
	So $\bmu_1$ (defined above as $\mu_1\times\bdelta^*\times\bdelta^*$) and $\bmu_2:=\mu_{12}\times\bdelta^*$ are the only two ergodic invariant probability measure of $\{\BS(t)\}$.
	The family
	$\left\{\check \Pi_t^{\bs}(\cdot):=\dfrac1t\int_0^t\PP_{\bs}\left\{(X(u),Y(u),Z(u))\in\cdot\right\}du, t\geq 0\right\}$
	is tight in $\R^3_+$ and
	 any weak-limit of $\check\Pi^t_{\bs}$ as $t\to\infty$ must be an invariant probability measure of $\{\BS(t)\}$, that is, the weak-limit has the form $p\bmu_1+(1-p)\bmu_{12}$ for some $p\in[0,1]$; see e.g \cite[Theorem 9.9]{ethier2009markov}.
	 We show that $p$ must be $0$.
	Assume that
	$\check\Pi_{t_k}^{\bs}$ converges weakly to $p\bmu_1+(1-p)\bmu_{12}$ as $t_k\uparrow\infty$ for some subsequence $\{t_k\}_{k=1}^\infty$.
	Then, we have
	$$
	\begin{aligned}
	\lim_{k\to\infty}&\int_{\R^3_+}\left(F_1(u,v)u-F_2(v,w)w-\alpha_{2}-\frac{\sigma_2^2}2\right)d\check\Pi_{t_k}^{\bs}
	\\
	&=\int_{\R^3_+}\left(F_1(u,v)u-F_2(v,w)w-\alpha_{2}-\frac{\sigma_2^2}2\right)(pd\bmu_1+(1-p)d\bmu_{12}).
	\end{aligned}
	$$
	Note that
	$$
	\int_{\R^3_+}\left(F_1(u,v)u-F_2(v,w)w-\alpha_{2}-\frac{\sigma_2^2}2\right)d\bmu_1=\lambda_1,
	$$
and
		$$
	\int_{\R^3_+}\left(F_1(u,v)u-F_2(v,w)w-\alpha_{2}-\frac{\sigma_2^2}2\right)d\bmu_{12}=0,
	$$
	which can be proved in the same manner as \cite[Lemma 3.4]{HN16}.
	As a result, we have
		$$
	\lim_{k\to\infty}\frac{\E_{\bs}\ln Y(t_k)}{t_k}=\lim_{k\to\infty}\int_{\R^3_+}(F_1(u,v)u-F_2(v,w)w-\alpha_{2}-\frac{\sigma_2^2}2)d\check\Pi_{t_k}^{\bs}=p\lambda_1.
	$$
	If $p>0$ then we end up with $\lim_{k\to\infty}\E_{\bs}\ln Y(t_k)=\infty$, which contradicts \eqref{e1-thm1}.
	Thus, $p$ must be $0$.
	As a result, for $\bs\in\R^{3,\circ}_+$, $\bnu_{12}$ is the unique weak-limit.
	
	Let $R_\eps>0$ such that $\mu_{12}\big([R_{\eps}^{-1},R_{\eps}]^2\big)> 1-\eps.$
	By the Heine-Borel covering theorem, there exists $(x_1,y_1),\cdots, (x_l,y_l)$ such that
	$[R_{\eps}^{-1},R_{\eps}]^2$ is covered by the union of disks centered at $(x_k,y_k)$ with radius $\frac12\varsigma_{x_k,y_k,\eps}$, $k=1,\cdots,n$; where $\varsigma$ is determined as in Lemma \ref{lm2}.
	Then, for any $\wdt\bs\in [R_{\eps}^{-1},R_{\eps}]^2\times(0,\frac12\varsigma_{\min})$ with $\varsigma_{\min}=\min_{k=1,\cdots,l}\{\varsigma_{x_k,y_k,\eps}\}$, there exists
	$k_{\wdt\bs}\in\{1,\cdots,l\}$ such that
	$$(\wdt x-x_{k_{\wdt\bs}})^2+(\wdt y-y_{k_{\wdt\bs}})^2+\wdt z^2\leq\varsigma_{\min}^2.$$
	Thus, we have
	\begin{equation}\label{e3-thm3}
	\PP_{\wdt\bs}\left\{\lim_{t\to\infty}\frac{\ln Z(t)}t=\lambda_2<0\right\}>1-\eps, \;\forall \wdt\bs\in [R_{\eps}^{-1},R_{\eps}]^2\times(0,\varsigma_{\min}).
	\end{equation}

On the other hand, since $\mu_{12}([R_{\eps}^{-1},R_{\eps}]^2)>1-\eps$,  there exists a  $\check T=\check T(\bs,\eps)>0$ such that
	$$\check\Pi^{\check T}_{\bs}([R_{\eps}^{-1},R_{\eps}]^2\times(0,\varsigma_{\min}))>1-2\eps,$$
	or equivalently,
	$$\dfrac1{\check T}\int_0^{\check T}\PP_{\bs}\{\BS(t)\in ([R_{\eps}^{-1},R_{\eps}]^2\times(0,\varsigma_{\min}))\}dt>1-2\eps.$$
	As a result,
	$$\PP_{\bs}\{\hat\tau\leq\check T\}>1-2\eps,$$
	where $\hat\tau=\inf\{t\geq 0: \BS(t)=(X(t),Y(t),Z(t))\in [R_{\eps}^{-1},R_{\eps}]^2\times(0,\varsigma_{\min})\}$.
Therefore,	using the strong Markov property and \eqref{e3-thm3},
	we deduce that
	\begin{equation}
		\PP_{\bs}\left\{ \lim_{t\to\infty}\frac{\ln Z(t)}{t}=\lambda_2\right\}\geq (1-\eps)(1-2\eps)\geq 1-3\eps \text{ given } \bs\in\R^{3,\circ}_+.
	\end{equation}
	Letting $\eps\to0$ we obtain the desired result.
\end{proof}
\section{Proof of Theorem \ref{thm4}}\label{s:pers}
The proof of Theorem \ref{thm4} will follow the idea from \cite{benaim2022stochastic}.
We will need the following estimates from \cite[Lemma 4.6]{benaim2022stochastic}.
\begin{lm}\label{lm3.0}
	Let $1<p\leq 2$. There exists $c_p>0$ such that for any $a>0$ and $x\in\R$ we have
	\beq\label{lm3.0-e1}
	|a+x|^{p}\leq
		a^{p}+pa^{p-1}x+c_p|x|^{p}.
	\eeq
	Moreover,  there exists $d_{p, b}>0$ depending only on $p,b>0$ such that if $x + a \geq 0$ then
\beq\label{lm3.0-e3}
(a+x)^{p}-b(a+x)^{p-1} \leq
a^{p}+ pa^{p-1} x-\frac{b}2 a^{p-1}+c_{b,p}(|x|^{p}+1).
\eeq
It follows straightforwardly from \eqref{lm3.0-e1} that
	for a random variable $R$ and a constant $c>0$, there exists $\tilde K_c >0$ such that
	\beq\label{lm3.0-e2}
	\E |R+c|^{p}\leq
		c^{p}+pc^{p-1}{\E R}+\tilde K_c\E |R|^{1+p} .
	\eeq
	\end{lm}
In this section, let
$\gamma_2>0$, $\gamma_3>0$ be such that
$$
L(\gamma_2\vee\gamma_3)\leq \frac12\min\{\alpha_1,\alpha_2-\alpha_4,\alpha _3- \alpha_5\}
\text{ and }
\gamma_2\lambda_1-\gamma_3 \left(\alpha_3+\frac{\sigma_3^2}2\right)>0,$$
and set
$$\rho:=\frac12\left[\left(\gamma_3\lambda_2\right)\vee \left(\gamma_2\lambda_1-\gamma_3\left(\alpha_3+\frac{\sigma_3^2}2\right) \right)\right]>0.$$
Pick $c_1>0$ such that  $y+  z-\gamma_2\ln y-\gamma_3\ln z+c_1\geq 0$ for any $(y,z)\in\R^{2,\circ}_+$
and consider
$$
V(\bs)=x+y+  z-\gamma_2\ln y-\gamma_3\ln z+c_1\geq 0,\bs\in\R^{3,\circ}_+.
$$
Then, because $F_1, F_2$ are bounded by $L$ and $(\gamma_1\vee\gamma_2)L\leq \min\{\alpha_1,\alpha_2-\alpha_4,\alpha _3- \alpha_5\}$, we have
\begin{equation}\label{LV}
	\begin{aligned}
\op V(\bs)=&\left(\Lambda-\alpha_1 x+(\alpha_4-\alpha_2)y+(\alpha _5- \alpha_3)z\right)
\\&+\gamma_2 \left(F_1(x,y)x-  F_2(y,z)z-\alpha_2-\frac{\sigma_2^2}2\right)+\gamma_3 \left( F_2(y,z)z-\alpha_3-\frac{\sigma_3^2}2\right)\\
\leq& A_V-\frac12\min\{\alpha_1,\alpha_2-\alpha_4,\alpha _3- \alpha_5\}(x+y+z)\leq \1_{\{|\bs|\leq M\}}A_V-\alpha_m V(\bs),
\end{aligned}
\end{equation}
for some positive constants $A_V, M $ and $\alpha_m$.

Let $q_0$ be as in Theorem \ref{thm1} and $A_V$, $\alpha_m$, $\rho$ as above. Let
$n^\diamond>0$ be such that
\begin{equation}\label{eqn0}
(n^\diamond -1)\alpha_m-2^{q_0-1}A_V\geq \frac\rho2.
\end{equation}

The following lemma gives us estimates for $\Lom V$ when the solution starts in a neighborhood of the boundary.
\begin{lm}\label{lm3.1}
	There exist $T^\diamond>0, \delta>0$ such that
	$$
	\E_\bs\int_0^T \op V(\BS(s))ds \leq -\rho T,$$
	for any  $T\in [T^\diamond, n^\diamond T^\diamond]$, $\bs\in\R^{3,\circ}_+$, $|\bs|\leq M$, and $\dist(\bs,\partial\R^{3,\circ}_+)\leq\delta$.
	
\end{lm}
\begin{proof}
	On the boundary, there are only two invariant probability $\bmu_1:=\mu_1\times\bdelta^*\times\bdelta^*$ and $\bmu_{12}:=\mu_{12}\times\bdelta^*$.
	In view of Theorem \ref{thm1} we can deduce the following claims:
	\begin{enumerate}
		\item[(C1)]	$(u+v+w)^{q_0}$ is integrable with respect to either $\bmu_1$ and $\bmu_{12}$
		and
\begin{equation}\label{e0-lm3.1}
		\int_{\R^3_+} \left(\Lambda-\alpha_1 u+(\alpha_4-\alpha_2)v+(\alpha _5- \alpha_3)w\right)d\bmu=0,\; \bmu\in\{\bmu_1,\bmu_{12}\},
\end{equation}
		and
\begin{equation}\label{e1-lm3.1}
		\int_{\R^3_+} \left(F_1(u,v)u-\alpha_2-\frac{\sigma_2^2}2\right)d\bmu_{12}=0.
\end{equation}
		(The proof is similar to that of \cite[Lemma 3.4]{HN16}.)
		\item[(C2)] $
		\{\check\Pi_t^{\bs}: t\geq 1, |\bs|\leq M\}
		$ is tight and all its weak-limits, as $t\to\infty$, must be invariant measures of $(X(t),Y(t),Z(t))$.
		(See e.g. \cite[Theorem 9.9]{ethier2009markov}.)
		\item[(C3)]
		For a sequence of bounded initial points $\{\bs_k\in\R^3_+\}$ and an increasing sequence $T_k\to\infty$ as $k\to\infty$,
		if $\{\check\Pi_{T_k}^{\bs}\}$ converges to $\mu$ as $T_k$ tends to $\infty$ then
		$$
		\lim_{k\to\infty} \int_{\R^3_+}h(\bs)\check\Pi_{T_k}^{\bs}(d\bs)=\int_{\R^3_+}h(\bs)\mu(d\bs)
		$$
		for any continuous functin $h(\bs)$ satisfying $h(\bs)\leq C_h(1+x+y)^{q}$ for some $C_h>0, 0<q<q_0$. (See \cite[Lemma 3.5]{HN16} for a similar proof.)
	\end{enumerate}
Next, we get from \eqref{LV}, \eqref{e0-lm3.1} and \eqref{e1-lm3.1} that
\begin{equation}\label{e3-lm3.1}
\int_{\R^3_+}\op V(\bs)d\bmu_{12}=-\gamma_3 \left(\int_{\R_{3}+} F_2(u,w)wd\bmu_{12}-\alpha_3-\frac{\sigma_3^2}2\right)=-\gamma_3\lambda_2\leq -2\rho,
\end{equation}
and that
\begin{equation}\label{e4-lm3.1}
\begin{aligned}
\int_{\R^3_+}\op V(\bs)\bmu_1(\bs)=&-\gamma_2\left(\int_{\R_{3}+} F_1(u,v)ud\bmu_{1}-\alpha_2-\frac{\sigma_2^2}2\right)-\gamma_3 \left(\int_{\R_{3}+} F_2(v,w)wd\bmu_{1}-\alpha_3-\frac{\sigma_3^2}2\right)\\
=&-\gamma_2\lambda_1+\gamma_3\left(\alpha_3+\frac{\sigma_3^2}2\right)\leq-2\rho.
\end{aligned}\end{equation}

Now, we claim that there exists $T^\diamond=T^\diamond(M)>0$ such that if $\bs\in\partial\R^3_+$ and $|\bs|\leq M$ then
\begin{equation}\label{e2-lm3.1}
\E_\bs\int_0^T \op V(X(\bs))d\bs= \int_{\R^3_+}\op V(\bs)d\check\Pi_{T}^{\bs}\leq -\frac32\rho T.
\end{equation}
Indeed, assuming the contrary, there exists a sequence $\{\bs_k\}\subset \partial\R^3_+$ such that $|\bs_k|\leq M$ and a sequence $T_k\uparrow \infty$ such that
$$
\E_x\frac1{T_k}\int_0^{T_k} \op V(X(s))ds>-\frac32\rho.
$$
Because of Claim (C2), there exist subsequences, which we still denote by $\{\bs_k\}$ and $\{T_k\}$ for convenience, such that
$\check\Pi_{T_k}^{\bs_k}$ converges to an invariant probability measure $\bmu$ as $k\to\infty$.
Because $\partial\R^3_1$ is an invariant set of the process $(X(t),Y(t),Z(t))$,
$\bmu$ must be a convex combination of $\bmu_1$ and $\bmu_{12}$.
Thus, in view of \eqref{e3-lm3.1} and \eqref{e4-lm3.1}, we have
$$
\int_{\R^3_+}\op V(\bs)\bmu(d\bs)<-2\rho.
$$
On the other hand, we have from Claim (C3) that
$$
\int_{\R^3_+}\op V(\bs)\bmu(d\bs)=\lim_{k\to\infty}\int_{\R^3_+}\op V(\bs)\check\Pi_{T_k}^{\bs_k}\geq-\frac32\rho.
$$
The contradiction shows the existence of $T^\diamond$ satisfying \eqref{e2-lm3.1}.

Then, by the Feller-Markov property of the process $(X(t),Y(t),Z(t))$ and the uniform boundedness \eqref{e1-thm1}, we can show that there exists $\delta>0$ such that
$$\E_\bs\int_0^T \op V(X(\bs))d\bs\leq -\rho T,\;\forall T\in[T^\diamond, n^\diamond T^\diamond],$$
for any $\bs\in\R^3_+$ satisfying $|\bs|\leq M$.
\end{proof}
Now, we are ready to establish a kind of drift condition that will help us establish the ergodicity of the underlying systems and obtain the rate of convergence.
\begin{prop}\label{lm3.3}
	Let $q$ be any number in the interval $(1,q_0)$,
	and $U(\bs)=1+|\bs|_1:=1+x+y+z$ for $\bs=(x,y,z)$.
	There is $\kappa^\diamond>0$ and $C_\diamond, C^\diamond >0$ such that
	$$
	\E_\bs [C_\diamond|\BS(n^\diamond T^\diamond)|^q+V^q(\BS(n^\diamond T^\diamond))]\leq C_\diamond U^q(x)+V^q(x)-\kappa^\diamond [C_\diamond U^q(\bs)+V^q(\bs)]^{\frac{q-1}q} + C^\diamond.
	$$
\end{prop}
\begin{proof}
	First we assume that $1<q\leq 2$.
	In the sequel, $C^\diamond$ is a generic constant depending on $T^\diamond, M, n^\diamond$ but independent of $x\in\R_{++}^n$.
	$C^\diamond$ can differ from line to line.
	Suppose $X(0)=x$. We have from It\^o's formula that
	$$
	V(X(t))=V(x)+\int_0^t \op V(X(s))ds+\wdt h(t)
	$$
	Here
	$$
	\wdt h(t):=\int_0^t (\sigma_1 X(s)dW_1(s)+\sigma_2 Y(s)dW_2(t)+ \sigma_3 Z(s)dW_3(s)-\gamma_2\sigma_2dW_2(s)-\gamma_3\sigma_3dW_3(s))
	$$
	 is a martingale with quadratic variation given by
	\beq\label{e3-lm3.3}\langle \wdt h(t)\rangle=\int_0^t \left(\sigma_1^2X^2(s)+\sigma_2^2 (Y(s)-\gamma_2)^2+\sigma_3^2(  Z(s)-\gamma_3)^2\right)ds\leq K \int_0^t U^2(\BS(s))ds,
	\eeq
	for some constant $K=K(\sigma_1,\sigma_2,\sigma_3,\gamma_2,\gamma_3)$.
	
	Because $\op V(\bs)\leq A_V$, we have $$V(X(T))=V(x)+\int_0^T \op V(X(s))ds+\wdt h(T)\leq V(x)+A_V T+\wdt M(T).$$  Applying \eqref{lm3.0-e2} yields
	\beq\label{e6-lm3.3}
	\begin{aligned}
		\E_\bs [V(\BS(T))]^q
		\leq &  V^q(\bs) + qA_V T V^{q-1}(\bs)+ C^\diamond(1+|\bs|_1)^q, \quad T\leq n^\diamond T^\diamond.
	\end{aligned}
	\eeq
	where $|\bs|_1=x+y+z$.
		%
	On the other hand, since $|\op V(\bs)|\leq K_0 (|\bs|_1+1),\forall \bs\in\R^3_+$ for some constant $K_0$, we deduce from It\^o's isometry and Hölder's inequality  that
	\beq\label{e4-lm3.3}
	\E_\bs\left|\int_0^t LV(\BS(s))ds\right|^q+\E_\bs \left|\wdt h(t)\right|^q \leq C^\diamond (|\bs|_1+1)^q, \quad\forall  t\leq n^\diamond T^\diamond,\bs\in\R^{3,\circ}_+.
	\eeq
	It follows from \eqref{e4-lm3.3} and \eqref{lm3.0-e2} that
	\beq\label{e0-lm3.3}
	\begin{aligned}
		\E_\bs [V(\BS(t))]^q\leq&  V^q(\bs) + q\left[\E_\bs \int_0^t \op V(\BS(s))ds\right] V^{q-1}(\bs)+ C^\diamond\E_\bs\left|\int_0^t \op V(\BS(s))ds+\wdt h(t)\right|^q\\
		\leq &  V^q(\bs) + q\left[\E_\bs \int_0^t \op V(\BS(s))ds \right] V^{q-1}(\bs)+ C^\diamond(1+|\bs|_1)^q, \quad\forall t\leq n^\diamond T^\diamond .
	\end{aligned}
	\eeq
	Thus, if $|\bs|_1\leq M $ and $ \dist(\bs,\partial \R_+^3)\leq \delta$, we have  $\E_\bs \int_0^t \op V(\BS(s))ds\leq -\rho t$, $t\in[T^\diamond , n^\diamond T^\diamond ]$. As a result,
	\beq\label{e5-lm3.3}
	\begin{aligned}
		\E_\bs [V(\BS(T))]^q
		\leq &  V^q(\bs) - q\rho T V^{q-1}(\bs)+ C^\diamond (1+|\bs|_1)^q, \quad T\in[T^\diamond , n^\diamond T^\diamond], |\bs|_1\leq M.
	\end{aligned}
	\eeq
	Noting that $V(\bs)$ is bounded on the set $\{\bs\in\R_{+}^3: |\bs|_1\leq M, \dist(\bs,\partial \R_{+}^3)\geq\delta\}$, it follows from
	 \eqref{e5-lm3.3} and \eqref{e6-lm3.3} for $|\bs|_1 \leq M$ that
	\beq\label{e7-lm3.3}
	\begin{aligned}
		\E_\bs [V(\BS(T))]^q
		\leq &  V^q(\bs) - q\rho T V^{q-1}(\bs)+ C^\diamond ,\quad \forall T\in[T^\diamond , n^\diamond T^\diamond].
	\end{aligned}
	\eeq
	
		Define 	 $$\zeta=\inf\{t\geq 0: X(t)+Y(t)+Z(t)\leq M\}\wedge (n^\diamond T^\diamond).$$
	From now on, we suppose that $|\bs_1| \leq M$. For $t\leq \zeta$, we deduce from \eqref{LV} that
	\beq\label{e9-lm3.3}
	V(\BS(t))= V(\bs)+\int_0^t\op V(\BS(s))ds +\wdt h(t)\leq V(\bs)- \alpha_mt+\wdt h(t).
	\eeq
	We have from \eqref{e7-lm3.3}, \eqref{e9-lm3.3}, \eqref{lm3.0-e3}, and the strong Markov property of $X(t)$ that
	\beq\label{e1-lm3.3}
	\begin{aligned}
		\E_\bs& \left[\1_{\{\zeta \leq T^\diamond (n^\diamond -1)\}} V^q(\BS(n^\diamond T^\diamond ))\right]\\[0.5ex]
		\leq &\E_\bs \left[\1_{\{\zeta \leq T^\diamond (n^\diamond -1)\}} \left[V^q(\BS(\zeta))+C^\diamond \right] \right] -\E_\bs\left[\1_{\{\zeta \leq T^\diamond (n^\diamond -1)\}}q\rho(n^\diamond T^\diamond -\zeta)V^{q-1}(\BS(\zeta))\right]\\[0.5ex]
		\leq& \E_\bs \left[\1_{\{\zeta \leq T^\diamond (n^\diamond -1)\}} (V(\bs)+\wdt h(\zeta))^q+C^\diamond  \right] -q\rho T^\diamond \E_\bs\left[\1_{\{\zeta \leq T^\diamond (n^\diamond -1)\}}(V(\bs)+\wdt h(\zeta))^{q-1}\right]\\[0.5ex]
		\leq& \E_\bs\left[\1_{\{\zeta \leq T^\diamond (n^\diamond -1)\}}\left( V^q(\bs)-\frac{q\rho T^\diamond }2 V^{q-1}(\bs)+ q\wdt h(\zeta)V^{q-1}(\bs)+C^\diamond (|\wdt h(\zeta)|^q+1)\right)\right].
	\end{aligned}
	\eeq
	If $T^\diamond (n^\diamond -1)\leq \zeta \leq T^\diamond n^\diamond $, we have
	\begingroup
	\allowdisplaybreaks
	\begin{align*}
		\E_\bs &\left[\1_{\{\zeta \geq T^\diamond (n^\diamond -1)\}} V^q(\BS(n^\diamond T^\diamond ))\right]\\[0.5ex]
		\leq &\E_\bs \left[\1_{\{\zeta \geq T^\diamond (n^\diamond -1)\}} V^q(\BS(\zeta))+C^\diamond \right] +qA_V\E_\bs\left[\1_{\{\zeta \geq T^\diamond (n^\diamond -1)\}}(n^\diamond T^\diamond -\zeta)V^{q-1}(\BS(\zeta))\right] \\[0.5ex]
		&\text{ (thanks to \eqref{e6-lm3.3} and the strong Markov property)}\\[0.5ex]
		\leq& \E_\bs \left[\1_{\{\zeta \geq T^\diamond (n^\diamond -1)\}} [(V(\bs)+\wdt h(\zeta)- \alpha_m\zeta)^q+C^\diamond ]\right] +qA_V T^\diamond \E_\bs\left[\1_{\{\zeta \geq T^\diamond (n^\diamond -1)\}}(V(\bs)+\wdt h(\zeta)- \alpha_m\zeta)^{q-1}\right] \\[0.5ex]
		&\text{ (because of \eqref{e9-lm3.3})}\\[0.5ex]
		\leq& \E_\bs\left[\1_{\{\zeta \geq T^\diamond (n^\diamond -1)\}}\left( V^q(\bs)-q \alpha_m\zeta V^{q-1}(\bs)+q\wdt h(\zeta) V^{q-1}(\bs) + C^\diamond \left(|\wdt h(\zeta)|+1\right)^q\right)\right]\\[0.5ex]
		&+ 2qA_V T^\diamond \E_\bs\left[\1_{\{\zeta \geq T^\diamond (n^\diamond -1)\}} \left( V^{q-1}(\bs)+|\wdt h(\zeta)|^{q-1}\right)\right]\\[0.5ex]
		& \text{ (applying \eqref{lm3.0-e1} and the inequality $|x+y|^{q-1}\leq 2(|x|^{q-1}+|y|^{q-1})$)} \\[0.5ex]
		\leq & \E_\bs \left[ \1_{\{\zeta \geq T^\diamond (n^\diamond -1)\}} \left( V^q(\bs)-\frac{q\rho T^\diamond }2 V^{q-1}(\bs)+q \wdt h(\zeta)V^{q-1}(\bs) +C^\diamond \left(|\wdt h(\zeta)|+1\right)^q \right)\right] \\
		& \text{ (since } \alpha_m \zeta\geq \alpha_mT^\diamond (n^\diamond -1))\geq \left(2A_V+\frac{\rho}2\right)T^\diamond).\\[-2ex]
		& \stepcounter{equation}\tag{\theequation}\label{e2-lm3.3}
	\end{align*}
	\endgroup
	As a result, by adding \eqref{e1-lm3.3} and \eqref{e2-lm3.3} and noting that $\E_\bs \wdt h(\zeta)=0$, we have
	\beq\label{e10-lm3.3}
	\begin{aligned}
		\E_\bs V^q(\BS(n^\diamond T^\diamond ))\leq&  V^q(\bs)-q\frac\rho2 T^\diamond  V^{q-1}(\bs)+ C^\diamond  \E_\bs (|\wdt h(\zeta)|+1)^q\\
		\leq& V^q(\bs)-q\frac\rho2 T^\diamond  V^{q-1}(\bs) + C^\diamond  U^{q}(\bs),
	\end{aligned}
	\eeq
	where the inequality $\E_\bs (|\wdt h(\zeta)|+1)^q\leq C^\diamond  U^{q}(\bs)$ comes from an application of
	the Burkholder-Davis-Gundy Inequality, H\"older's inequality, \eqref{e3-lm3.3} and \eqref{e2-thm1}. From \eqref{e1-thm1}, we have
	\beq\label{e11-lm3.3}
	\begin{aligned}
		\E_\bs U^q(\BS(n^\diamond T^\diamond ))\leq&  U^q(\bs)-\left(1-e^{-k_{2q} n^\diamond T^\diamond }\right) U^q(\bs) + \frac{k_{1q}}{k_{2q}}.
	\end{aligned}
	\eeq
	Combining \eqref{e10-lm3.3} and \eqref{e11-lm3.3}, we get that
	\beq\label{e12-lm3.3}
	\E_\bs \left[V^q(\BS(n^\diamond T^\diamond ))+C_\diamond U^q(\BS(n^\diamond T^\diamond ))\right]
	\leq V^q(\bs)+C_\diamond U^q(\bs)-\kappa^\diamond  [V^q(\bs)+C_\diamond U^q(\bs)]^{(q-1)/q} + C^\diamond ,
	\eeq
	for some $ \kappa^\diamond >0, C^\diamond >0$ and sufficiently large $C_\diamond$.
\end{proof}
\begin{proof}[Proof of Theorem \ref{thm4}]
Having Proposition \ref{lm3.3}, the proof of Theorem \ref{thm4} is standard. Because of the nondegeneracy of the diffusion process and \eqref{e12-lm3.3}, we have from \cite[Theorem 3.6]{jarner2002polynomial} that
\begin{equation}\label{thm4-e11}
\lim_{k\to\infty}k^{ q-1}\|P_{kn^\diamond T^\diamond}(\bs,\cdot)-\mu^\diamond(\cdot)\|_{TV}=0,\; 1\leq  q<q_0
\end{equation}
where $\mu^\diamond$ is an invariant probability measure of the Markov chain $\{\BS(n^\diamond T^\diamond )\}$, which is also an invariant probability measure of the Markov process $\{\BS(t),t\geq 0\}$ due to the uniqueness of invariant probability measures.
Because $\|P_t(\bs,\cdot)-\mu^\diamond(\cdot)\|_{TV}$ is decreasing in $t$, we can easily deduce \eqref{thm4-e1} from \eqref{thm4-e11}.
A similar argument can be found in \cite[Proof of Theorem 1.1]{benaim2022stochastic} or \cite[Theorem 2.2]{dieu2016classification}. The proof is complete.
\end{proof}
\textbf{Acknowledgments:} The research has been done under the research project QG.22.10 “Asymptotic behaviour of mathematical models in ecology” of Vietnam National University, Hanoi for Nguyen Trong Hieu. A. Hening and D. Nguyen acknowledge
support from the NSF through the grants DMS-2147903  and DMS-1853467 respectively.
N. Nguyen acknowledges support from an AMS-Simons travel grant.

\bibliographystyle{amsalpha}
\bibliography{LV}
\end{document}